\PassOptionsToPackage{unicode}{hyperref}
\PassOptionsToPackage{naturalnames}{hyperref}
\documentclass[a4paper,10pt]{amsart}

\usepackage{amscd,amsmath,amsthm,amssymb,amscd}
\usepackage{latexsym,amsbsy,mathrsfs}
\usepackage[enableskew,vcentermath]{youngtab}

\usepackage{txfonts}

\usepackage[bbgreekl]{mathbbol}

\usepackage{amsfonts,amssymb,amscd,verbatim, mathbbol}
\usepackage{bm,multicol}
\usepackage[pdftex,linkcolor=blue,citecolor=blue,backref=page]{hyperref}

\usepackage[misc]{ifsym}

\def\NewTheorem#1{%
  \newaliascnt{#1}{equation}%
  \newtheorem{#1}[#1]{#1}%
  \aliascntresetthe{#1}%
  \expandafter\def\csname #1autorefname\endcsname{#1}%
}

\def\cP{\mathcal{P}}

\newcommand{\ve}{\varepsilon}
\newcommand{\fa}{\mathfrak a}

\newcommand{\ft}{\mathfrak t}

\newcommand{\ff}{\mathfrak f}

\newcommand{\fu}{\mathfrak u}
\newcommand{\fs}{\mathfrak s}

\newcommand{\fp}{\mathfrak p}
\newcommand{\fq}{\mathfrak q}

\newcommand{\wh}{\widehat}

\newcommand{\ot}{\otimes}

\newcommand{\N}{\mathbb N}
\newcommand{\Z}{\mathbb Z}
\newcommand{\e}{\epsilon}

\newcommand{\rF}{\mathrm{F}}

\newcommand{\rp}{\mathrm{p}}

\newcommand{\ba}{\begin {eqnarray}}
\newcommand{\ea}{\end {eqnarray}}
\newcommand{\baa}{\begin {eqnarray*}}
\newcommand{\eaa}{\end {eqnarray*}}
\newcommand{\be}{\begin {equation}}
\newcommand{\ee}{\end {equation}}
\newcommand{\bee}{\begin {equation*}}
\newcommand{\eee}{\end {equation*}}

\newcommand{\te}[1]{\textnormal{{#1}}}

\def \<{{\langle}}
\def \>{{\rangle}}

\newcommand\blam{{\boldsymbol\lambda}}

\newcommand\bmu{{\boldsymbol\mu}}
\newcommand\bnu{{\boldsymbol\nu}}

\DeclareMathOperator\Shape{Shape}

\let\<=\langle
\let\>=\rangle
\def\({\big(}
\def\){\big)}

\def\t{\mathfrak{t}}
\def\s{\mathfrak{s}}
\def\u{\mathfrak{u}}
\def\v{\mathfrak{v}}

\def\lam{\lambda}

\def\Sym{\mathfrak{S}}

\newcommand\HH{\mathscr{H}}

\def\P{\mathscr{P}}

\def\bQ{\mathbf{Q}}
\def\bu{\mathbf{u}}
\def\mL{\mathcal{L}}
\def\fm{\mathfrak{m}}

\def\mff{\mathfrak{f}}

\def\mf{\mathfrak{f}}

\def\SStd{\mathop{\rm Std}\nolimits^2}

\DeclareMathOperator\res{res}

\DeclareMathOperator\Std{Std}

\numberwithin{equation}{section}

\title[Standard symmetrizing forms and Murphy bases]{Standard symmetrizing forms and Murphy bases of the cyclotomic Hecke algebras}

\author[]{Jun Hu\textsuperscript{1}}
\author[]{Huansheng Li\textsuperscript{2, \Letter}}

\thanks{\Letter\, (Corresponding author) Huansheng Li,\quad E-mail: 3120225736@bit.edu.cn}

\numberwithin{equation}{section}
\newtheorem{prop}[equation]{Proposition}
\newtheorem{thm}[equation]{Theorem}
\newtheorem{cor}[equation]{Corollary}

\newtheorem{lem}[equation]{Lemma}

\theoremstyle{definition}
\newtheorem{dfn}[equation]{Definition}
\theoremstyle{remark}


\begin{document}
\maketitle

{\centering{\small \textsuperscript{1}  Key Laboratory of Algebraic Lie Theory and Analysis of Ministry of Education, \\
School of Mathematics and Statistics, Beijing Institute of Technology, Beijing, 100081, P.R.~China\\
E-mail: junhu404@bit.edu.cn\\}}
\vspace{5pt}
	
{\centering{\small \textsuperscript{2} Key Laboratory of Mathematical Theory and Computation in Information Security, \\
 School of Mathematics and Statistics,
	Beijing Institute of Technology,
	Beijing, 100081, P.R.~China\\
E-mail: 3120225736@bit.edu.cn\\}}

\begin{abstract} In this paper we compute the explicit value of the standard symmetrizing form of the cyclotomic Hecke algebra of type $G(\ell,1,n)$ on each Murphy basis element. We do this for both the non-degenerate cyclotomic Hecke algebra of type $G(\ell,1,n)$ and the degenerate cyclotomic Hecke algebra of type $G(\ell,1,n)$. As an application, we obtain some equations on the coefficients
which occur in the transition matrix between the Murphy basis elements and the seminormal basis elements of the cyclotomic Hecke algebra of type $G(\ell,1,n)$.
\end{abstract}

\setcounter{tocdepth}{1}

\section{Introduction}

Let $R$ be an integral domain, $1\neq q\in R^\times$ and ${\mathbf Q}:=(Q_{1},Q_2,\dots, Q_{\ell})\in R^\ell$. The non-degenerate cyclotomic Hecke algebras $\HH_{\ell,n}^R(q;\bQ)$ of type $G(\ell,1,n)$ with Hecke parameter $q$ and cyclotomic parameters $Q_{1},Q_2,\dots, Q_{\ell}$ were first introduced in \cite[Definition 3.1]{AK1}, \cite[Definition 4.1]{BM:cyc} and \cite[before Proposition 3.2]{C} as certain deformations of the group ring $R[W_{\ell,n}]$ of the complex reflection group $W_{\ell,n}$ of type $G(\ell,1,n)$. By definition, $\mathscr{H}_{\ell,n}=\HH_{\ell,n}^R(q;\bQ)$ is the unital associative $R$-algebra with generators $T_0,T_1,\dots,T_{n-1}$ and the following defining relations:
\begin{align}
&(T_0-Q_1)(T_0-Q_2)\cdots(T_0-Q_\ell)=0;\\
&T_0T_1T_0T_1=T_1T_0T_1T_0;\\
\label{nondegquadraticrela}&(T_i-q)(T_i+1)=0,\quad \forall\,1\leq i\leq n-1;\\
\label{nondegbraidrela1}&T_iT_j=T_jT_i,\ \ \forall\,1\leq i<j-1<n-1;\\
\label{nondegbraidrela2}&T_iT_{i+1}T_i=T_{i+1}T_iT_{i+1},\ \ \forall\,1\leq i<n-1.
\end{align}
If $\ell=1$ or $\ell=2$, then the cyclotomic Hecke algebra $\HH_{\ell,n}^R(q;\bQ)$ becomes the Iwahori-Hecke algebra of type $A_{n-1}$ or type $B_n$ respectively.

The symmetric group $\Sym_n$ on $\{1,2,\dots,n\}$ is generated by $\{\sigma_i:=(i,i+1)\,|\,1\leq i<n\}$. A word $w=\sigma_{i_{1}}\sigma_{i_{2}}\cdots\sigma_{i_{k}}$, where $1\leq i_1,i_2,\dots,i_k\leq n-1$, is called a reduced expression of $w$ if $k$ is minimal; in this case we say that $w$ has length $k$ and we write $\ell(w)=k$. Note that $T_w:=T_{i_1}T_{i_2}\cdots T_{i_k}$ depends only on $w$ but not on the choice of the reduced expression of $w$ because of the braid relations holding in $\HH_{\ell,n}$.

The $R$-subalgebra of $\mathscr{H}_{\ell,n}$ generated by $T_1,\dots,T_{n-1}$ is isomorphic to the Iwahori-Hecke algebra of type $A_{n-1}$, which has a standard basis $\{T_w\,|\, w\in\Sym_n\}$. We set $$
\mL_{1}=T_{0},\quad \mL_{k+1}=q^{-1}T_{k}L_{k}T_{k},\ \ \forall\,1\leq k\leq n-1 .
$$
These elements commute with each other, and are called the {\bf Jucys-Murphy operators} of $\HH_{\ell,n}^R(q;\bQ)$.

\begin{lem}\label{AK basis}\te{(}Ariki--Koike \cite[(3.10)]{AK1}\te{)}
The algebra $\mathscr{H}_{\ell,n}$ is free as an $R$-module with basis
$$\{\mL_1^{c_1}\mL_2^{c_2}\cdots \mL_n^{c_n}T_w\mid
       w\in\Sym_n,\ 0\leq c_i\leq \ell-1,\ \forall\,1\leq i\leq n\}.$$
In particular, the $R$-rank of $\mathscr{H}_{\ell,n}$ is equal to $\ell^nn!$.
\end{lem}

\begin{dfn}\label{ssf} For any $w\in \Sym_n$ and integers $0\leq c_1,c_2,\dots,c_n<\ell,$ we define
$$
\tau_{R}^{\rm{MM}}(\mL_1^{c_1}\mL_2^{c_2}\cdots \mL_n^{c_n}T_w):=\begin{cases}
1, &\text{if\ \,$w=1$\ and\ $c_1=c_2=\cdots=c_n=0$},\\
0, &\text{otherwise.}
\end{cases}
$$
We extend $\tau_{R}^{\rm{MM}}$ linearly to an $R$-linear function on $\HH_{\ell,n}$.
\end{dfn}
One can show directly that the $R$-linear function $\tau_{R}^{\rm{MM}}$ is symmetric, i.e., $\tau_{R}^{\rm{MM}}(xy)=\tau_{R}^{\rm{MM}}(yx)$, $\forall\,x,y\in \HH_{\ell,n}$, see \cite{MM} and \cite{HHL}. We call $\tau_R^{\rm{MM}}$ the {\bf{standard symmetrizing form}} on $\HH_{\ell,n}$. If furthermore, $(Q_{1},Q_2,\dots, Q_{\ell})\in (R^\times)^\ell$, then by a result of Malle-Mathas \cite{MM}, the $R$-linear function $\tau_{R}^{\rm{MM}}$ on $\HH_{\ell,n}$ is non-degenerate, which makes $\HH_{\ell,n}$ into a symmetric algebra over $R$.

We use $\P_{\ell,n}$ to denote the set of $\ell$-partitions of $n$. For each $\blam\in\P_{\ell,n}$, we use $\Std(\blam)$ to denote the set of standard $\blam$-tableaux. Each $\t\in\Std(\blam)$ can be viewed as a bijective map
$\t: [\blam]\rightarrow\{1,2,\dots,n\}$ from the Young diagram $[\blam]:=\{(r,c,l)\mid 1\leq l\leq\ell,\ r\geq 1,\ 1\leq c\leq\lam_r^{(l)}\}$ to $\{1,2,\dots,n\}$.

In \cite{DJM}, Dipper-James-Mathas constructed a cellular basis $\{\fm_{\s\t}\,|\, \s,\t\in\Std(\blam),\ \blam\in\P_{\ell,n}\}$ for $\HH_{\ell,n}$
in the sense of \cite{GL}. This basis generalizes the well-known Murphy basis for the Iwahori-Hecke algebra associated to $\Sym_n$ (i.e., the level one case). We call it {\bf the Murphy basis for the non-degenerate cyclotomic Hecke algebra $\HH_{\ell,n}$}. When $\HH_{\ell,n}$ is semisimple, Mathas also constructed a seminormal basis $\{\mf_{\s\t}\,|\,\s,\t\in\Std(\blam),\ \blam\in\P_{\ell,n}\}$ for $\HH_{\ell,n}$. We refer the readers to Sections 2.1 and 3.2 for their precise definitions.

A critical challenge in the study of $\HH_{\ell,n}$ lies in characterizing the transition matrix between the Murphy basis and the seminormal basis. The explicit transformation between the Murphy basis and the seminormal basis—whose entries encode essential structural correlations—lacks a systematic description. This gap persists partly due to the absence of precise computations of trace forms on
 the Murphy basis, as the values of the trace form on seminormal basis elements can be easily expressed in terms of the gamma coefficients and the Schur functions. Without such exact values, the transition matrix remains elusive, obstructing efforts to unify the combinatorial (Murphy basis) and representation-theoretic (seminormal basis) perspectives of Hecke algebras.

Another pressing problem concerns the integral basis of the Hecke algebra's cocenter. In \cite{HS25}, the first author and Lei Shi have constructed an integral basis for the cocenter of $\HH_{\ell,n}$. However, it is very hard to apply that integral basis of the cocenter because it is given by a subset of Bremke-Malle's standard basis $\{T_w\}$ whose definition depends on the choice of the reduced expression of $w$ and is not directly related to the cellular structure of $\HH_{\ell,n}$. We really expect an integral basis for the cocenter of $\HH_{\ell,n}$ which is directly related to the cellular structure of $\HH_{\ell,n}$. The Murphy basis, with its inherent integral and recursive properties, offer a promising avenue for constructing such a basis—but this requires precise knowledge of trace forms on the Murphy basis to ensure integrality and linear independence.

\begin{dfn}\label{keydefn} Let $k\in\N,\ \blam=(\lam^{(1)},\dots,\lam^{(\ell)})\in\P_{\ell,k}$ and $\t\in\Std(\blam)$.
Let $1\leq k(\t)\leq k$ be the unique integer such that
\begin{equation*}\label{ntDef}\t^{-1}(k)=(\t^{\blam})^{-1}(k(\t)).
\end{equation*}
For each integer $1\leq j\leq n$, we let $r(\t,j),\,c(\t,j)$ and $l(\t,j)$ be the integers such that
\begin{equation*}\label{coordinate}
\t^{-1}(j)=(r(\t,j),c(\t,j),l(\t,j))\in [\blam].
\end{equation*}
%where $1\leq l(\t,j)\leq \ell$ and $1\leq c(\t,j)\leq\lam^{(l(\t,j))}_{r(\t,j)}$.
\end{dfn}

The following theorem is the first main result of this paper.

\begin{thm}\label{mainthm1} Let $\blam\in\P_{\ell,n}$ and $\s,\t\in\te{Std}(\blam).$ Then
$$\tau_R^{\te{MM}}(\fm_{\s\t})=
\begin{cases}
q^{\sum\limits_{k=1}^n(k-k(\t_{\downarrow\leq k}))}\bigg(\prod\limits_{j=1}^n\prod\limits_{i=l(\t,j)+1}^{\ell}(-Q_i)\bigg), & \mbox{if\ \ }\s=\t, \\
0, & \mbox{if\ \ }\s\neq \t.
\end{cases}$$
\end{thm}

Let $\bu=(u_1,u_2,\dots,u_\ell)\in R^{\ell}$. The degenerate cyclotomic Hecke algebra $H_{\ell,n}=H_{\ell,n}^R(\bu)$ of type $G(\ell,1,n)$ is the unital associative $R$-algebra with generators $s_{1} ,\dots,s_{n-1},L_1, \dots,L_n$ which satisfy certain relations, see Definition \ref{degenerate}. In \cite{AMR}, Ariki, Mathas and Rui constructed a cellular basis $\{m_{\s\t}\,|\, \s,\t\in\Std(\blam),\ \blam\in\P_{\ell,n}\}$ for $H_{\ell,n}$. This basis generalizes the well-known Murphy basis for the symmetric group algebra $R[\Sym_n]$ (i.e., the level one case). We call it {\bf the Murphy basis for the degenerate cyclotomic Hecke algebra $H_{\ell,n}$}. Brundan-Kleshchev \cite{BK08} introduced a symmetrizing form $\tau_R^{\te{BK}}$ on $H_{\ell,n}$ which makes it into a symmetric algebra over $R$. We refer the readers to Section 2.2 for their precise definitions. The following theorem is the second main result of this paper.

\begin{thm}\label{mainthm2} Let $\blam\in\P_{\ell,n}$ and $\s,\t\in\te{Std}(\blam).$ Then
$$\tau_R^{\te{BK}}(m_{\s\t})=
\begin{cases}
1, & \mbox{if\ \,}\s=\t\te{\ \,and\ \,} l(\t,i)=1,\ \forall\,1\leq i\leq n, \\
0, & \mbox{\text{otherwise}}.
\end{cases}$$
\end{thm}

The main idea of the proof of Theorems \ref{mainthm1} and \ref{mainthm2} comes from \cite{HLL},
where we identify the standard symmetrizing form on $\HH_{\ell,n}$ (resp., on $H_{\ell,n}$) with the composition of a series of descending map $\ve_k: \HH_{\ell,k}\rightarrow\HH_{\ell,k-1}$  (resp., $\e_k: H_{\ell,k}\rightarrow H_{\ell,k-1}$),
where each descending map $\ve_k$ (resp., $\e_k$) is given by the cyclotomic Mackey decomposition of $\HH_{\ell,n}$ (resp., of $H_{\ell,n}$).
Instead of computing the value of the standard symmetrizing form via its original definition, we shall show that
each $\ve_n(\fm_{\s\t})$ (resp., $\e_n(m_{\s\t})$) is zero or a scalar multiple of
$\fm_{\s_{\downarrow\leq (n-1)}\t_{\downarrow\leq (n-1)}}$ (resp., of $m_{\s_{\downarrow\leq (n-1)}\t_{\downarrow\leq (n-1)}}$),
and we shall precisely compute this scalar.

The content of the paper is organised as follows.
In Section 2, we recall some preliminary results on the non-degenerate cyclotomic Hecke algebra and the degenerate cyclotomic Hecke algebra of type $G(\ell,1,n)$.
In Section 3, we use the cyclotomic Mackey decomposition of the non-degenerate cyclotomic Hecke algebra of type $G(\ell,1,n)$
to show in Proposition \ref{MainProp1}  that each $\ve_n(\fm_{\s\t})$ is zero or a scalar multiple of $\fm_{\s_{\downarrow\leq (n-1)}\t_{\downarrow\leq (n-1)}}$, and thus get a proof of the first main result Theorem \ref{mainthm1} of this paper.
Then we apply Proposition \ref{MainProp1} and Theorem \ref{mainthm1}
to derive some equations in Corollaries \ref{maincor11}, \ref{maincor12} and \ref{maincor13}  on the coefficients
which occur in the transition matrix between the Murphy basis elements and the seminormal basis elements. In Section 4, we generalize all the results obtained in Section 3 to the degenerate cyclotomic Hecke algebra of type $G(\ell,1,n)$.

\bigskip\bigskip
\centerline{Acknowledgements}
\bigskip

The research was support by the National Natural Science Foundation of China (No. 12431002).
\bigskip

\bigskip

\section{Preliminary}\label{preliminary}

In this section we shall recall some preliminary knowledge on the non-degenerate cyclotomic Hecke algebra and the degenerate cyclotomic Hecke algebra of type $G(\ell,1,n)$.

\subsection{Non-degenerate cyclotomic Hecke algebra of type $G(\ell,1,n)$}
In this subsection we shall collect some basic results on the non-degenerate cyclotomic Hecke algebra $\HH_{\ell,n}=\HH_{\ell,n}^R(q;\bQ)$,
whose definition has been given in Section 1. We shall also introduce some basic combinatorial notions.

\begin{lem}\text{\rm (\cite[(1.1)]{Ma})}\label{nondeg-L-comm}
Suppose that $1\le i\le n-1$ and $1\le j\le n$. Then:
\begin{enumerate}
\item $\mL_i$ and $\mL_j$ commute;
\item $T_i$ and $\mL_j$ commute if $i\ne j-1,\, j$;
\item $T_i$ commutes with $\mL_i\mL_{i+1}$ and $\mL_i+\mL_{i+1}$;
\item $T_i\mL_i=\mL_{i+1}T_i-(q-1)\mL_{i+1}$ and $T_{i}\mL_{i+1}=\mL_{i}T_{i}+(q-1)\mL_{i+1}$;
\item If $a\in R$ and $i\ne j$ then $T_i$ commutes with $(\mL_1-a)(\mL_2-a)\cdots(\mL_j-a)$.
\end{enumerate}
\end{lem}

\begin{lem}\label{nondegMadecomlem}  \te{(}\cite[Lemma 2.2]{HLL}\te{)}
\begin{enumerate}
  \item There are some canonical $(\mathscr{H}_{\ell,n-1},\mathscr{H}_{\ell,n-1})$-bimodule isomorphisms:
$$\mathscr{H}_{\ell,n-1} \otimes_{\mathscr{H}_{\ell,n-2}}\mathscr{H}_{\ell,n-1}\cong \mathscr{H}_{\ell,n-1} T_{n-1} \mathscr{H}_{\ell,n-1} ,
\quad a\ot b\mapsto aT_{n-1}b ;$$
$$\mathscr{H}_{\ell,n-1}\cong \mathscr{H}_{\ell,n-1} \mL_n^k  ,\quad c\mapsto c \mL_n^k  \quad (0\leq k \leq \ell-1).$$
  \item As $(\mathscr{H}_{\ell,n-1},\mathscr{H}_{\ell,n-1})$-bimodules, we have
$$\mathscr{H}_{\ell,n}=\mathscr{H}_{\ell,n-1} T_{n-1} \mathscr{H}_{\ell,n-1} \bigoplus \bigoplus_{k=0}^{\ell-1} \mathscr{H}_{\ell,n-1} \mL_n^{k}.$$
\end{enumerate}
\end{lem}

For any ring homomorphism $\varphi: B\rightarrow A$ between two rings $A,B$, we regard $A$ as an $(B,B)$-bimodule via $\varphi$. For any $f\in\{a\in A\mid ab=ba,\ \forall\,b\in B\}$, we define $$
\mu_f: A\otimes_{B}A\rightarrow A,\ \ \sum_{i}a_i\otimes b_i\mapsto \sum_{i}a_ifb_i,\ \ \text{where\ \,$a_i,b_i\in A,\ \forall\,i$.}
$$

By Lemma \ref{nondegMadecomlem}, we see that for any $h\in \mathscr{H}_{\ell,n},$ there are unique elements $$
\pi^{(n)}(h)\in \mathscr{H}_{\ell,n-1} \otimes_{\mathscr{H}_{\ell,n-2}} \mathscr{H}_{\ell,n-1},\quad p^{(n)}_k(h)\in\mathscr{H}_{\ell,n-1},\,\,k=0,1,\ldots,\ell-1, $$  such that
\begin{align}\label{nondegMadecomequation}
h=\mu_{T_{n-1}}(\pi^{(n)}(h))+\sum_{k=0}^{\ell-1}p^{(n)}_k(h)\mL_n^k.
\end{align}
We call (\ref{nondegMadecomequation}) {\bf the cyclotomic Mackey decomposition} of $\mathscr{H}_{\ell,n}$.

Let $\ast$ be the unique $R$-algebra anti-automorphism of $\HH_{\ell,n}$ which fixes all of its Hecke generators $T_0,T_1,\dots,T_{n-1}$. It is easy to verify that $\mL_j^\ast=\mL_j$ for any $1\leq j\leq n$. For $1\leq k\leq n$, we define the following map:
$$\ve_k:\mathscr{H}_{\ell,k}\rightarrow\mathscr{H}_{\ell,k-1},\ \ h\mapsto p^{(k)}_{0}(h). $$

\begin{lem}\label{veAst} For any $h\in \mathscr{H}_{\ell,k}$ and $h_1\in\mathscr{H}_{\ell,k-1}$, we have that $$
\ve_k(hh_1)=\ve_k(h)h_1,\quad \ve_k(h_1h)=h_1\ve_k(h),\quad \ve_k(h^*)=\ve_k(h)^* .
$$
\end{lem}
\begin{proof} This follows from Lemmas \ref{AK basis} and \ref{nondeg-L-comm} and the uniqueness of the cyclotomic Mackey decomposition (\ref{nondegMadecomequation}).
\end{proof}

We define an $R$-linear map
$$\tau_n^R:=\varepsilon_1 \circ \varepsilon_2 \circ \cdots\circ \varepsilon_n: \mathscr{H}_{\ell,n} \rightarrow  R.$$

\begin{lem}\label{nondegtauoperation} \te{(}\cite[Lemma 2.6]{HLL}\te{)} Let $w\in\Sym_n$ and $\mathbf{c}:=(c_1,c_2,\dots,c_n)\in\Z^{n},$ where
$0\leq c_i \leq \ell-1$ for each $i$. Then we have
$$\tau_n^R(\mL_1^{c_1}\mL_2^{c_2}\cdots \mL_n^{c_n}T_w)=\delta_{(\mathbf{c},w),(\underline{0},1)},$$
where $\underline{0}:=(0,0,\dots,0)\in\Z^{n}$.
In particular, the map $\tau_n^R$ coincides with the standard symmetrizing form $\tau_{R}^{\te{MM}}$.
\end{lem}

A composition of $m$ is a sequence $\lam=(\lam_{1},\lam_{2},\dots)$ of non-negative integers such that $|\lam|:=\sum_{i\geq 1}\lam_{i}=m$. We set $\ell(\lam):=\max\{j\geq 1\,|\,\lam_j\neq 0\}$.
Given a composition $\lam=(\lam_{1},\lam_{2},\dots,\lam_k)$ of $m$ with $\lam_i>0,\ \forall\,1\leq i\leq k$, the corresponding Young subgroup of $\Sym_m$ is defined to be \begin{equation}\label{Young110}
\Sym_\lam:=\Sym_{\{1,2,\dots,\lam_1\}}\times\Sym_{\{\lam_1+1,\lam_1+2,\dots,\lam_1+\lam_2\}}
\times\cdots\times\Sym_{\{(\sum_{j=1}^{k-1}\lam_j)+1,\dots,m\}} .
\end{equation}
A composition $\lam=(\lam_{1},\lam_{2},\dots)$ is called a partition if $\lam_{1}\geq\lam_{2}\geq\cdots$.

An $\ell$-partition of $n$ is a sequence $\blam=(\lam^{(1)},\dots,\lam^{(\ell)})$ of partitions such that $|\lam^{(1)}|+\cdots+|\lam^{(\ell)}|=n$. %We write $\blam\vdash n$, and call $\lam^{(i)}$ the $i$th component of $\blam$.
Let $\P_{\ell,n}$ be the set of $\ell$-partitions of $n$. The Young diagram of an $\ell$-partition $\blam\in\P_{\ell,n}$ is the set
$$[\blam]:=\bigl\{(r,c,l)\bigm|1\leq c\leq \lam_{r}^{(l)},\ 1\leq l\leq \ell\bigr\}.$$
The Young diagram $[\blam]$ of an $\ell$-partition is viewed as a sequence of Young diagrams $[\lam^{(l)}],\ 1\leq l\leq \ell$ of partitions. The element $(r,c,l)$ is called a node. A node $(r,c,l)\in [\blam]$ is called a removable node of $\blam$ if $[\blam]\setminus \{(r,c,l)\}$ is the Young diagram of an $\ell$-partition $\bmu$. In this case, we call $(r,c,l)$ an addable node of $\bmu$. We use ${\te{Rem}}(\blam)$ and ${\te{Add}}(\blam)$ to denote the set of removable nodes of $\blam$ and the set of addable nodes of $\blam$ respectively.

A $\blam$-tableau is a bijection $\t: [\blam]\rightarrow\{1,2,\dots,n\}$. If $\t$ is a $\blam$-tableau, we write $\Shape(\t):=\blam$. A $\blam$-tableau $\t$ is standard, if $\t(i,j,l)\leq\t(a,b,l)$ whenever $i\leq a$, $j\leq
b$ and $1\leq l\leq \ell$. Let $\Std(\blam)$ be the set of standard $\blam$-tableaux.
%A standard tableau $\t\in\Std(\blam)$ can be viewed as a sequence $\t=(\t^{(1)},\cdots,\t^{(\ell)})$, where each $\t^{(k)}$ is an injective map $\t^{(k)}: [\blam^{(k)}]\hookrightarrow\{1,2,\cdots,n\}$ satisfying $\t^{(k)}(i,j,k)\leq\t^{(k)}(a,b,k)$ whenever $i\leq a$, $j\leq
%b$, and $\{1,2,\cdots,n\}=\sqcup_{k=1}^{\ell}\im(\t^{(k)})$. We call $\t^{(i)}$ the $i$th component of $\t$ for each $1\leq i\leq \ell$.
For any $\t\in\Std(\blam)$, we set ${\te{Add}}(\t):={\te{Add}}(\blam)$ and ${\te{Rem}}(\t):={\te{Rem}}(\blam)$. For any $\blam\in\P_{\ell,n}$, we define $\SStd(\blam):=\{(\s,\t)\,|\,\s,\t\in\Std(\blam)\}$. We set
$\Std(\P_{\ell,n}):=\sqcup_{\blam\in\P_{\ell,n}}\Std(\blam)$ and $\SStd(\P_{\ell,n}):=\sqcup_{\blam\in\P_{\ell,n}}\SStd(\blam)$.

Let $\blam=(\lam^{(1)},\lam^{(2)},\dots,\lam^{(\ell)})$ be an $\ell$-partition of $n$. Let $\fa=(\fa_1,\fa_2,\dots,\fa_\ell),$ where $\fa_1:=0,\  \fa_s=|\lam^{(1)}|+|\lam^{(2)}|+\cdots+|\lam^{(s-1)}|$ for $2\leq s\leq \ell.$ We identify $\blam$ with the composition $$
(\lam_1^{(1)},\lam_2^{(1)},\dots,\lam_{\ell(\lam^{(1)})}^{(1)},\lam_1^{(2)},\lam_2^{(2)},\dots,\lam_{\ell(\lam^{(2)})}^{(2)},\dots) $$ of $n$, and define
the Young subgroup $\Sym_{\blam}$ as before.

Let $\blam\in\P_{\ell,n}$ and $\t\in\Std(\blam)$.
We let $d(\t)$ be the unique element in $\Sym_n$ such that $\t=\t^\blam d(\t)$. Following \cite{Mur95} and \cite{DJM}, we define $$
x_{\blam}:=\sum_{w\in\Sym_\blam}T_w,\quad
\fu_\blam^+:=\prod_{s=2}^{\ell}\prod_{k=1}^{\fa_s}(\mL_k-Q_s). $$
\begin{dfn}[{\cite{Mur95}, \cite{DJM}, \cite{Ma}}] Let $\blam\in\P_{\ell,n}$ and $\s,\t\in\Std(\blam)$. We define $$
\fm_\blam:=x_\blam \fu_\blam^+,\quad \fm_{\s\t}:=T_{d(\s)}^*\fm_\blam T_{d(\t)}.$$
\end{dfn}

Let $\blam\in\P_{\ell,n}$ and $\s,\t\in\Std(\blam)$. It follows from Lemma \ref{nondeg-L-comm} that
$x_\blam \fu_\blam^+=\fu_\blam^+ x_\blam,$ hence
\begin{align}\label{nondegmst*}
\fm_{\s\t}^*=\fm_{\t\s}.
\end{align}

For any $\blam,\bmu\in\P_{\ell,n}$, we write $\blam\unrhd\bmu$ if for any $1\leq s\leq\ell$ and $i\geq 1$ we have $$\sum_{t=1}^{s-1}|\lam^{(t)}|+\sum_{j=1}^{i}\lam^{(s)}_{j}\geq \sum_{t=1}^{s-1}|\mu^{(t)}|+\sum_{j=1}^{i}\mu^{(s)}_{j}.$$ We write $\blam\rhd\bmu$ if $\blam\unrhd\bmu$ and $\blam\neq\bmu$. Then $\P_{\ell,n}$ becomes a poset under the dominance order ``$\unrhd$''.

Let $\t\in\te{Std}(\cP_{\ell,n}).$ We use $\t_{\downarrow\leq k}$ to denote the subtableau of $\t$ that contains the numbers $\{1,2,\dots,k\}$.  Note that for each $0\leq k\leq n$, $\Shape(\t_{\downarrow\leq k})$ is an $\ell$-partition. For any $\blam\in\P_{\ell,n}$, we set $\blam_{\downarrow\leq(n-1)}:=\te{Shape}((\ft^\blam)_{\downarrow\leq (n-1)})$.

For any $\s, \t\in\Std(\cP_{\ell,n})$, we write $\s\unrhd\t$ if for each $1\leq k\leq n$,
${\te{Shape}}(\s_{\downarrow\leq k})\unrhd{\te{Shape}}(\t_{\downarrow\leq k})$. We write $\s\rhd\t$ if $\s\unrhd\t$ and $\s\neq\t$. The dominance ordering can be extend to $\Std^2(\P_{\ell,n})$ as follows: $(\s,\t)\unrhd(\u,\v)$ if $\s\unrhd\u$ and $\t\unrhd\v$. We write $(\s,\t)\rhd(\u,\v)$ if $(\s,\t)\unrhd(\u,\v)$ and $(\s,\t)\neq(\u,\v)$.

\begin{lem}[\cite{Mur95}, \cite{DJM}]\label{cellular1} The set $\{\fm_{\s \t}\,|\, \s, \t \in \Std(\bm{\lambda}),\ \bm{\lambda} \in \P_{\ell,n}\}$, together with the poset $(\P_{\ell,n},\unrhd)$ and the anti-involution ``$\ast$'', form a cellular $R$-basis of $\HH_{\ell,n}$ in the sense of \cite{GL}.
\end{lem}

\subsection{Degenerate cyclotomic Hecke algebra of type $G(\ell,1,n)$}
In this subsection we shall collect some definitions and results on the degenerate cyclotomic Hecke algebra $H_{\ell,n}$.

We let $R$ be an integral domain and let $\bu=(u_1,u_2,\dots,u_\ell)\in R^{\ell}$.
\begin{dfn}\label{degenerate} The degenerate cyclotomic Hecke algebra $H_{\ell,n}=H_{\ell,n}^R(\bu)$ of type $G(\ell,1,n)$ is the unital associative $R$-algebra with generators $s_{1},s_2,\dots,s_{n-1},L_1,L_2,\dots,L_n$ and the following defining relations:
\begin{align}
	&(L_1-u_1)(L_1-u_2)\cdots(L_1-u_\ell)=0;\\
\label{degquadraticrela}	&s_{i}^{2}=1,\quad \forall\,1\leq i\leq n-1;\\
\label{degbraidrela1}	&s_{i}s_{j}=s_{j}s_{i},\ \ \forall\,1\leq i<j-1<n-1;\\
\label{degbraidrela2}	&s_is_{i+1}s_i=s_{i+1}s_is_{i+1},\ \ \forall\,1\leq i<n-1;\\
\label{deg-L-comm1} &L_iL_k=L_kL_i,\ \  s_iL_l=L_ls_i,\ \ 1\leq i<n,\ 1\leq k,l\leq n,\ l\neq i,i+1;\\
\label{deg-L-comm2} &L_{i+1}=s_iL_is_i+s_i,\ \ 1\leq i<n .
\end{align}
\end{dfn}
The elements $L_1,L_2,\dots,L_n$ are called the Jucys-Murphy elements of $H_{\ell,n}$. The following results are easy consequences of Definition \ref{degenerate}.

\begin{lem}\label{deg-L-comm3}
Suppose that $1\le i\le n-1$ and $1\le j\le n$. Then:
\begin{enumerate}
\item $s_iL_{i+1}=L_is_i+1$ and $ L_{i+1}s_i=s_iL_i+1;$
\item If $a\in R$ and $i\ne j$ then $s_i$ commutes with $(L_1-a)(L_2-a)\cdots(L_j-a)$.
\end{enumerate}
\end{lem}

Recall the cyclotomic Mackey decomposition of $H_{\ell,n}$ given in \cite[Lemma 7.6.1]{K}.

\begin{lem}\te{(}\cite[Lemma 7.6.1]{K}\te{)}\label{degMadecomlem}
\begin{enumerate}
 \item  $H_{\ell,n}$ is a free right $H_{\ell,n-1}$-module with basis
\begin{align*}
\{ L_n^c,\ L_j^{c}s_js_{j+1} \cdots s_{n-1}\mid 0\leq c \leq \ell-1,\ 1\leq j \leq n-1 \}.
\end{align*}
In particular, the set $\{L_1^{c_1}L_2^{c_2}\cdots L_n^{c_n}w\,|\,w\in\Sym_n,\ 0\leq c_j\leq \ell-1,\ \forall\, 1\leq j \leq n-1\}$ gives an $R$-basis of $H_{\ell,n}$.
  \item There are some canonical $(H_{\ell,n-1},H_{\ell,n-1})$-bimodule isomorphisms:
$$H_{\ell,n-1} \otimes_{H_{\ell,n-2}}H_{\ell,n-1}\cong H_{\ell,n-1} s_{n-1} H_{\ell,n-1} ,
\quad a\ot b\mapsto as_{n-1}b ;$$
$$H_{\ell,n-1}\cong H_{\ell,n-1} L_n^k  ,\quad c\mapsto c L_n^k  \quad (0\leq k \leq \ell-1).$$
\item   As $(H_{\ell,n-1},H_{\ell,n-1})$-bimodules, we have
$$H_{\ell,n}=H_{\ell,n-1} s_{n-1} H_{\ell,n-1} \bigoplus \bigoplus_{k=0}^{\ell-1} H_{\ell,n-1} L_n^{k}.$$
\end{enumerate}
\end{lem}

By Lemma \ref{degMadecomlem}, we see that for any $h\in H_{\ell,n},$ there are unique elements $\varpi^{(n)}(h)\in H_{\ell,n-1} \otimes_{H_{\ell,n-2}} H_{\ell,n-1},$ and ${\rm p}^{(n)}_k(h)\in H_{\ell,n-1},$ $k=0,1,\ldots,\ell-1$  such that
\begin{align}\label{degMadecomequation}
h=\mu_{s_{n-1}}(\varpi^{(n)}(h))+\sum_{k=0}^{\ell-1}{\rm p}^{(n)}_k(h)L_n^k.
\end{align}
For $1\leq k\leq n$, we define the following map:
$$\e_k:H_{\ell,k}\rightarrow H_{\ell,k-1},\ \ h\mapsto \rp^{(k)}_{\ell-1}(h). $$

Let $\ast$ be the unique $R$-algebra anti-automorphism of $H_{\ell,n}$ which fixes all of its defining generators $s_1,\dots,s_{n-1},L_1,\dots,L_n$. Similar to the non-degenerate case, we have the following result.

\begin{lem}\label{veAst2} For any $h\in H_{\ell,k}$ and $h_1\in H_{\ell,k-1}$, we have that $$
\e_k(hh_1)=\e_k(h)h_1,\quad \e_k(h_1h)=h_1\e_k(h),\quad \e_k(h^*)=\e_k(h)^* .
$$
\end{lem}
We define an $R$-linear map$$ t_n^R:=\epsilon_1 \circ \epsilon_2 \circ \cdots\circ \epsilon_n: H_{\ell,n} \rightarrow R.$$

\begin{dfn}\te{(}\cite{BK08}\te{)}
Let $\tau_R^{\te{BK}}$  be the unique $R$-linear map $\tau_R^{\te{BK}}: H_{\ell,n} \rightarrow  R$ which is defined on its $R$-basis by
\begin{align*}
\tau_R^{\te{BK}}(L_1^{c_1}L_2^{c_2}\cdots L_n^{c_n}w)=
\begin{cases}
1, & \mbox{if }c_1=c_2=\cdots=c_n=\ell-1\te{ and }w=1, \\
0, & \mbox{otherwise},
\end{cases}
\end{align*}
where $w\in\Sym_n $ and $0\leq c_i\leq \ell-1$ for $i=1,2,\dots,n$.
\end{dfn}

The $R$-linear function $\tau_{R}^{\rm{BK}}$ is a symmetrizing form on $H_{\ell,n}$, which makes $H_{\ell,n}$ into a symmetric algebra over $R$. We call $\tau_R^{\rm{BK}}$ the {\bf{standard symmetrizing form}} on $H_{\ell,n}$.

\begin{lem}\text{\rm (\cite[Lemma 2.11]{HLL})}\label{degtoperation}
Let $w\in\Sym_n$ and $\mathbf{c}:=(c_1,c_2,\dots,c_n)\in\Z^{n},$ where
$0\leq c_i \leq \ell-1$ for each $i$. Then we have
$$t_{n}^R(L_1^{c_1} L_2^{c_2}\cdots L_n^{c_n}w)=\delta_{(\mathbf{c},w),(\underline{\ell-1},1)},$$
where $\underline{\ell-1}:=(\ell-1,\ell-1,\dots,\ell-1)\in\Z^n$. In particular, the map $t_n^R$ coincides with the standard symmetrizing form $\tau_{R}^{\te{BK}}$.
\end{lem}

Let $\blam=(\lam^{(1)},\lam^{(2)},\dots,\lam^{(\ell)})$ be an $\ell$-partition of $n$. Following \cite{Mur95} and \cite{AMR}, we define $$
{\rm{x}}_{\blam}:=\sum_{w\in\Sym_\blam}w,\quad u_\blam^+:=\prod_{s=2}^{\ell}\prod_{k=1}^{\fa_s}(L_k-u_s),\quad m_\blam:={\rm{x}}_\blam u_\blam^+. $$
For any $\s,\t\in\te{Std}(\blam)$, we define $m_{\s\t}:=d(\s)^*m_\blam d(\t)$.
It follows from Lemma \ref{deg-L-comm3} that ${\rm{x}}_\blam u_\blam^+=u_\blam^+ {\rm{x}}_\blam,$ hence
\begin{align}\label{degmst*}
m_{\s\t}^*=m_{\t\s}.
\end{align}

\begin{lem}[\cite{Mur95}, \cite{AMR}]\label{cellular2} The set $\{m_{\s \t}\,|\, \s, \t \in \Std(\bm{\lambda}),\ \bm{\lambda} \in \P_{\ell,n}\}$, together with the poset $(\P_{\ell,n},\unrhd)$ and the anti-involution ``$\ast$'', form a cellular $R$-basis of $H_{\ell,n}$ in the sense of \cite{GL}.
\end{lem}
This basis $\{m_{\s\t}\}$ generalizes the well-known Murphy basis for the symmetric group algebra $R[\Sym_n]$ (i.e., the level one case). So we shall call it {\bf the Murphy basis for the degenerate cyclotomic Hecke algebra $H_{\ell,n}$}.

\subsection{Some key facts}

In this subsection, we will give some basic facts, which will be frequently used in the following two sections.

\begin{dfn} For any integers $1\leq k,c,d\leq n-1,$ we set
$$\beta_{c,k}:=
\begin{cases}
\sigma_c\sigma_{c+1}\cdots \sigma_{k}, & \mbox{if\ \ } c\leq k, \\
1, & \mbox{if\ \ }c>k,
\end{cases}\quad
  \gamma_{k,d}:=
  \begin{cases}
 \sigma_k\sigma_{k-1}\cdots \sigma_{d} , & \mbox{if\ \ }d\leq k,\\
   1, &\mbox{if\ \ } d>k.\\
  \end{cases} $$
\end{dfn}
Note that the elements $\beta_{c,n-1},\ 1\leq c\leq n-1$ and $1$ are minimal length left coset representatives of $\Sym_{n-1}$ in $\Sym_n$, and the elements $\gamma_{n-1,d},\ 1\leq d\leq n-1$ and $1$ are minimal length right coset representatives of $\Sym_{n-1}$ in $\Sym_n$.
For any $u,w,y\in\Sym_n$, we write $w=u\cdot y$ if $w=uy$ and $\ell(w)=\ell(u)+\ell(y)$.

Let $\blam\in \P_{\ell,n}$ and $\t\in\Std(\blam)$. Recall Definition \ref{keydefn} that
$1\leq n(\t)\leq n$ is the unique integer such that $\t^{-1}(n)=(\t^\blam)^{-1}(n(\t))$.
By definition, we have that \begin{align}\label{decompositionofdt}
d(\t)=\begin{cases}
d(\t_{\downarrow\leq (n-1)}), & \mbox{if\ \ } n(\t)=n, \\
 \beta_{n(\t),n-1}\cdot d(\t_{\downarrow\leq (n-1)}), & \mbox{if\ \ }n(\t)<n.
      \end{cases}
\end{align}

Recall that we have identified the $\ell$-partition $\blam\in\P_{\ell,n}$ with the composition  $$
(\lam_1^{(1)},\lam_2^{(1)},\dots,\lam_{\ell(\lam^{(1)})}^{(1)},\lam_1^{(2)},\lam_2^{(2)},\dots,\lam_{\ell(\lam^{(2)})}^{(2)},\dots) $$
of $n$. For each $1\leq j\leq\ell$, we define $$
\Sym_{\blam^{[j]}}:=\Sym_{\{\fa_j+1,\dots,\fa_j+\lam_1^{(j)}\}}
\times\cdots\times\Sym_{\{\fa_{j+1}-\lam_{\ell(\lam^{(j)})}^{(j)}+1,\dots,\fa_{j+1}\}},
$$
which is naturally isomorphic to $\Sym_{\lam^{(j)}}$. Let $w\in \Sym_\blam.$ By definition we can write
\begin{align}\label{decompositionofwinSymlam}
w=w_{1}\cdot w_{2}\cdot\ldots\cdot w_{\ell},\quad
\te{where\ \ }w_{i}\in\Sym_{\blam^{[i]}},\ \forall\,1\leq i\leq \ell.
\end{align}
For each $1\leq i\leq \ell$, we set $r_i:=\ell(\lam^{(i)})$. For each $1\leq i\leq \ell$ and $1\leq j\leq r_i,$ we define $a_{i,j}:=\fa_{i-1}+\lam^{(i)}_1+\cdots+\lam^{(i)}_{j}$ and $a_{i,0}:=\fa_{i-1}.$
Then for each $1\leq i\leq \ell,$ we can write
\begin{align}\label{decompositionofwinSymlam2}
w_{i}=w_{i,1}\cdot w_{i,2}\cdot\ldots \cdot w_{i,r_i},\quad
\te{where\ \ }w_{i,j}\in\Sym_{\{a_{i,j-1}+1,\dots,a_{i,j}\}},\ \forall\,1\leq j\leq r_i.
\end{align}
For each $1\leq i\leq \ell$ and $1\leq j\leq r_i,$ we consider the distinguished left $\Sym_{\{a_{i,j-1}+1,\dots, a_{i,j}-1\}}$-coset decomposition and the distinguished right $\Sym_{\{a_{i,j-1}+1,\dots, a_{i,j}-1\}}$-coset decomposition of $w_{i,j}$:
\begin{align}\label{decompositionofwlamr}
w_{i,j}\in \Sym_{\{a_{i,j-1}+1,\dots, a_{i,j}-1\}} \te{\quad or\quad}
w_{i,j}=\beta_{b_{i,j},a_{i,j}-1}\cdot w_{i,j}^{(1)}=w_{i,j}^{(2)}\cdot\gamma_{a_{i,j}-1,c_{i,j}},
\end{align}
%satisfying that
%\begin{align}\label{lengthstrictlyincreases}
%\ell(w_{i,j})=\ell(\beta_{b_{i,j},a_{i,j}-1})+\ell(w_{i,j}^{(1)})
%=\ell(w_{i,j}^{(2)})+\ell(\gamma_{a_{i,j}-1,c_{i,j}}),
%\end{align}
where $w_{i,j}^{(1)},w_{i,j}^{(2)}\in \Sym_{\{a_{i,j-1}+1,\dots, a_{i,j}-1\}}$ and $a_{i,j-1}+1\leq b_{i,j},c_{i,j}\leq a_{i,j}-1.$ The notations $\{a_{ij},b_{ij},c_{ij}\}$ will be frequently used in the following two sections without further explanation.

\begin{dfn}
For any $1\leq k\leq n-1,\ 1\leq m\leq n-k$ and $k\leq i_1<i_2<\dots<i_m\leq n-1,$ we set
$$\beta^{i_1,\dots,i_m}_{k,n-1}:=\sigma_{k}\cdots \wh{\sigma_{i_1}}\cdots \wh{\sigma_{i_m}}\cdots \sigma_{n-1},$$
where the notation ``$\wh{\sigma_{i_1}},\dots,\wh{\sigma_{i_m}}$'' means to remove the simple reflections $\sigma_{i_1},\dots, \sigma_{i_m}$ from $\sigma_k\sigma_{k+1}\cdots\sigma_{n-1}$.
\end{dfn}

\begin{dfn}
For each $w\in\Sym_{\{2,\dots,n\}}$ with reduced expression
$w=\sigma_{i_1}\sigma_{i_2}\cdots \sigma_{i_m},$ where $2\leq i_1,i_2,\dots,i_m\leq n-1,$ we define
$$w_{\downarrow}:=\begin{cases}
\sigma_{i_1-1}\sigma_{i_2-1}\cdots \sigma_{i_m-1}, & \mbox{if\ \,}m\geq 1, \\
1, & \mbox{if\ \,}m=0,\ \te{i.e.},\ w= 1,
                  \end{cases}$$
which is independent of the choice of the reduced expression of $w$ because $\Sym_{\{2,3,\dots,n\}}\cong\Sym_{\{1,2,\dots,n-1\}}$.
\end{dfn}

\begin{lem}\label{keytools}
Let $1\leq k\leq n-1$. In $\HH_{\ell,n},$ we have the following statements.
\begin{itemize}
\item[(1)] If $k<i\leq n-1$, then we have
$T_{\gamma_{n-1,k}}T_i=T_{i-1}T_{\gamma_{n-1,k}}. $
\item[(2)] If $1\leq m\leq n-k$ and $k\leq i_1<i_2<\dots<i_m\leq n-1,$ then we have
$$T_{\gamma_{n-1,k}}T_{\beta_{k,n-1}^{i_1,\dots,i_m}}=\mu_{T_{n-1}}(\pi_{k,n-1}^{i_1,\dots,i_m})
\te{\quad for some\quad} \pi_{k,n-1}^{i_1,\dots,i_m}\in \HH_{\ell,n-1}\ot_{\HH_{\ell,n-2}}\HH_{\ell,n-1}.$$
\item[(3)] $T_{\gamma_{n-1,k}}T_{\beta_{k,n-1}}=\mu_{T_{n-1}}(\pi_{k,n-1})+q^{n-k}$
for some $\pi_{k,n-1} \in \HH_{\ell,n-1}\ot_{\HH_{\ell,n-2}}\HH_{\ell,n-1}.$
\end{itemize}
\end{lem}

\begin{proof}
First, using (\ref{nondegbraidrela1}) and (\ref{nondegbraidrela2}) we have
\begin{align*}
T_{\gamma_{n-1,k}}T_i&=T_{n-1}T_{n-2}\cdots T_{i+1}T_{i}T_{i-1}T_{i}T_{i-2}\cdots T_k\\
&=T_{n-1}T_{n-2}\cdots T_{i+1}T_{i-1}T_{i}T_{i-1}T_{i-2}\cdots T_k\\
&=T_{i-1}T_{n-1}T_{n-2}\cdots T_{i+1}T_{i}T_{i-1}T_{i-2}\cdots T_k=T_{i-1}T_{\gamma_{n-1,k}},
\end{align*}
proving the part (1).

For proving (2), we use the downward induction on $k$.
If $k=n-1$, the result is obvious.
Now we suppose that $k<n-1$. In this case, if $i_1=k$, then the result follows from the repeated applications of the part (1).
On the other hand, if $i_1>k$, then using (\ref{nondegquadraticrela}), the part (1) and the induction hypothesis, we have
\begin{align*}
T_{\gamma_{n-1,k}}T_{\beta_{k,n-1}^{i_1,\dots,i_m}}&=T_{\gamma_{n-1,k}}T_k T_{\beta_{k+1,n-1}^{i_1,\dots,i_m}}
=T_{\gamma_{n-1,k+1}}((q-1)T_k+q) T_{\beta_{k+1,n-1}^{i_1,\dots,i_m}}\\
&=(q-1)T_{\gamma_{n-1,k}} T_{\beta_{k+1,n-1}^{i_1,\dots,i_m}}+qT_{\gamma_{n-1,k+1}} T_{\beta_{k+1,n-1}^{i_1,\dots,i_m}}\\
&=(q-1)T_{(\beta_{k+1,n-1}^{i_1,\dots,i_m})_{\downarrow}}T_{\gamma_{n-1,k}}+q \mu_{T_{n-1}}(\pi_{k+1,n-1}^{i_1,\dots,i_m}),
\end{align*}
where $\pi_{k+1,n-1}^{i_1,\dots,i_m}\in \HH_{\ell,n-1}\ot_{\HH_{\ell,n-2}}\HH_{\ell,n-1}.$ This proves (2).

For proving (3), we still use the downward induction on $k$.
If $k=n-1$, then the result follows from (\ref{nondegquadraticrela}).
When $k<n-1$, using (\ref{nondegquadraticrela}), the part (1) and the induction hypothesis, we have
\begin{align*}
T_{\gamma_{n-1,k}}T_{\beta_{k,n-1}}&=T_{\gamma_{n-1,k+1}}((q-1)T_{k}+q)T_{\beta_{k+1,n-1}}\\
&=(q-1)T_{\gamma_{n-1,k}} T_{\beta_{k+1,n-1}}+qT_{\gamma_{n-1,k+1}} T_{\beta_{k+1,n-1}}\\
&=(q-1)T_{\beta_{k,n-2}}T_{\gamma_{n-1,k}}+q(\mu_{T_{n-1}}(\pi_{k+1,n-1})+q^{n-k-1}),
\end{align*}
where $\pi_{k+1,n-1} \in \HH_{\ell,n-1}\ot_{\HH_{\ell,n-2}}\HH_{\ell,n-1}.$
This completes the proof.
\end{proof}

\begin{lem}\label{JucyMurphyCommutators}
Let $j_1,j_2\dots,j_n$ and $k$ be some integers such that
$ j_1,\dots,j_{n-1}\geq 0$, $0\leq j_n\leq \ell-1$ and $1\leq k\leq n-1$. Then, in $\HH_{\ell,n}$ we have
\begin{align*}
\Biggr(\prod_{r=1}^{n}\mL_r^{j_r}\Biggr) T_{\beta_{k,n-1}}
=T_{\beta_{k,n-1}} \Biggr(\prod_{r=1}^{n}\mL_{\beta_{k,n-1}^{-1}(r)}^{j_r}\Biggr)
+\sum_{\substack{1\leq m\leq n-k\\k\leq i_1<i_2<\dots<i_m\leq n-1}} T_{\beta^{i_1,\dots,i_m}_{k,n-1}}f_{i_1,\dots,i_m},
\end{align*}
where $f_{i_1,\dots,i_m}=\sum_{a=0}^{\ell-1}f_{i_1,\dots,i_m}^{(a)}\mL^a_n$
for some $f_{i_1,\dots,i_m}^{(a)}\in R[\mL_1,\dots,\mL_{n-1}]$.
\end{lem}

\begin{proof}
We use downward induction on $k.$
First, if $k=n-1$, then using \cite[Lemma 3.3]{MM}, we have that \begin{align}
\label{commcompu} &\quad\, \Biggr(\prod_{r=1}^{n}\mL_r^{j_r}\Biggr) T_{n-1}\\
\notag &=\Biggr(\prod_{r=1}^{n-2}\mL_r^{j_r}\Biggr)T_{n-1}\mL_{n-1}^{j_{n}}\mL_n^{j_{n-1}}
\pm (q-1)\Biggr(\prod_{r=1}^{n-2}\mL_r^{j_r}\Biggr)\sum_{c=\min\{j_{n-1},j_n\}+1}^{\max\{j_{n-1},j_n\}}\mL_{n}^{j_{n-1}+j_n-c}\mL_{n-1}^{c}\\
\notag &=T_{n-1}\Biggr(\prod_{r=1}^{n-2}\mL_r^{j_r}\Biggr)\mL_{n-1}^{j_{n}}\mL_n^{j_{n-1}}
\pm (q-1)\Biggr(\prod_{r=1}^{n-2}\mL_r^{j_r}\Biggr)\sum_{c=\min\{j_{n-1},j_n\}+1}^{\max\{j_{n-1},j_n\}}\mL_{n}^{j_{n-1}+j_n-c}\mL_{n-1}^{c},
\end{align}
hence the result holds in this case. Now we suppose that $k<n-1$. Then using Lemma \ref{nondeg-L-comm},
a similar computation as in (\ref{commcompu}) and the induction hypothesis, we have
\begin{align*}
\notag &\quad\,\Biggr(\prod_{r=1}^{n}\mL_r^{j_r}\Biggr) T_{\beta_{k,n-1}}\\
&=\Biggr(\prod_{r=1}^{k+1}\mL_r^{j_r}\Biggr)T_k \Biggr(\prod_{p=k+2}^{n}\mL_p^{j_p}\Biggr) T_{\beta_{k+1,n-1}}\\
&=\Bigg(T_{k}\Biggr(\prod_{r=1}^{k-1}\mL_r^{j_r}\Biggr)\mL_{k+1}^{j_{k}}\mL_{k}^{j_{k+1}}
\pm (q-1)\Biggr(\prod_{r=1}^{k-1}\mL_r^{j_r}\Biggr)\sum_{c=\min\{j_k,j_{k+1}\}+1}^{\max\{j_{k},j_{k+1}\}} \mL_{k+1}^{j_{k}+j_{k+1}-c}\mL_{k}^{c}\Biggr)\\
&\qquad \times\Biggr(\prod_{p=k+2}^{n}\mL_p^{j_p}\Biggr) T_{\beta_{k+1,n-1}}\\
&=T_{k}\Biggr(\prod_{r=1}^{k-1}\mL_r^{j_r}\Biggr)\mL_{k}^{j_{k+1}}\Biggr(\mL_{k+1}^{j_{k}}\prod_{p=k+2}^{n}\mL_p^{j_p}\Biggr) T_{\beta_{k+1,n-1}}\\
&\qquad \pm (q-1)\Biggr(\prod_{r=1}^{k-1}\mL_r^{j_r}\Biggr)\sum_{c=\min\{j_k,j_{k+1}\}+1}^{\max\{j_{k},j_{k+1}\}} \mL_{k+1}^{j_{k}+j_{k+1}-c}
 \mL_{k}^{c}\Biggr(\prod_{p=k+2}^{n}\mL_p^{j_p}\Biggr) T_{\beta_{k+1,n-1}}\\
&=T_{\beta_{k,n-1}} \Biggr(\prod_{r=1}^{n}\mL_{\beta_{k,n-1}^{-1}(r)}^{j_r}\Biggr)
+\sum_{\substack{1\leq m\leq n-k\\k\leq i_1<\dots<i_m\leq n-1}} T_{\beta^{i_1,\dots,i_m}_{k,n-1}}f_{i_1,\dots,i_m}
\end{align*}
for some $f_{i_1,\dots,i_m}=\sum_{a=0}^{\ell-1}f_{i_1,\dots,i_m}^{(a)}\mL^a_n$
with $f_{i_1,\dots,i_m}^{(a)}\in R[\mL_1,\dots,\mL_{n-1}]$.
This completes the proof.
\end{proof}

\begin{lem}\label{keytoolsdeg}
Let $1\leq k\leq n-1$. In $H_{\ell,n},$ we have the following statements.
\begin{itemize}
\item[(1)] If $k<i\leq n-1$, then we have
$\gamma_{n-1,k}s_i=s_{i-1}\gamma_{n-1,k}. $
\item[(2)] If $1\leq m\leq n-k$ and $k\leq i_1<i_2<\dots<i_m\leq n-1,$ then we have
$$\gamma_{n-1,k}\beta_{k,n-1}^{i_1,\dots,i_m}=\mu_{s_{n-1}}(\varpi_{k,n-1}^{i_1,\dots,i_m})
\te{\quad for some\quad} \varpi_{k,n-1}^{i_1,\dots,i_m}\in H_{\ell,n-1}\ot_{H_{\ell,n-2}}H_{\ell,n-1}.$$
\end{itemize}
\end{lem}

\begin{proof}
The lemma follows from (\ref{degquadraticrela})-(\ref{degbraidrela2}) and a similar argument as in the proof of Lemma \ref{keytools}.
\end{proof}

\begin{lem}\label{JucyMurphyCommutatorsdeg}
Let $j_1,j_2\dots,j_n$ and $k$ be some integers such that
$ j_1,\dots,j_{n-1}\geq 0$, $0\leq j_n\leq \ell-1$ and $1\leq k\leq n-1$. Then, in $H_{\ell,n}$ we have
\begin{align*}
\Biggr(\prod_{r=1}^{n}L_r^{j_r}\Biggr) \beta_{k,n-1}
=\beta_{k,n-1} \Biggr(\prod_{r=1}^{n}L_{\beta_{k,n-1}^{-1}(r)}^{j_r}\Biggr)
+\sum_{\substack{1\leq m\leq n-k\\k\leq i_1<i_2<\dots<i_m\leq n-1}} \beta^{i_1,\dots,i_m}_{k,n-1}f_{i_1,\dots,i_m},
\end{align*}
where $f_{i_1,\dots,i_m}=\sum_{a=0}^{\ell-1}f_{i_1,\dots,i_m}^{(a)}L^a_n$
for some $f_{i_1,\dots,i_m}^{(a)}\in R[L_1,\dots,L_{n-1}]$.
\end{lem}

\begin{proof} We use downward induction on $k.$
First, if $k=n-1$, then similarly to \cite[Lemma 3.3]{MM}, we have that \begin{align}
\label{commcompudeg} &\quad\, \Biggr(\prod_{r=1}^{n}L_r^{j_r}\Biggr) s_{n-1}\\
\notag &=\Biggr(\prod_{r=1}^{n-2}L_r^{j_r}\Biggr)s_{n-1}L_{n-1}^{j_{n}}L_n^{j_{n-1}}
\pm \Biggr(\prod_{r=1}^{n-2}L_r^{j_r}\Biggr)\sum_{c=\min\{j_{n-1},j_n\}}^{\max\{j_{n-1},j_n\}-1}L_{n-1}^{j_{n-1}+j_n-c-1}L_{n}^{c}\\
\notag &=s_{n-1}\Biggr(\prod_{r=1}^{n-2}L_r^{j_r}\Biggr)L_{n-1}^{j_{n}}L_n^{j_{n-1}}
\pm \Biggr(\prod_{r=1}^{n-2}L_r^{j_r}\Biggr)\sum_{c=\min\{j_{n-1},j_n\}}^{\max\{j_{n-1},j_n\}-1}L_{n-1}^{j_{n-1}+j_n-c-1}L_{n}^{c},
\end{align}
hence the result holds in this case. Now we suppose that $k<n-1$. Then using Lemma \ref{deg-L-comm3},
a similar computation as in (\ref{commcompudeg}) and the induction hypothesis, we have
\begin{align*}
\notag &\quad\,\Biggr(\prod_{r=1}^{n}L_r^{j_r}\Biggr) \beta_{k,n-1}\\
&=\Biggr(\prod_{r=1}^{k+1}L_r^{j_r}\Biggr)s_k \Biggr(\prod_{p=k+2}^{n}L_p^{j_p}\Biggr) \beta_{k+1,n-1}\\
&=\Biggr(s_{k}\Biggr(\prod_{r=1}^{k-1}L_r^{j_r}\Biggr)L_{k+1}^{j_{k}}L_{k}^{j_{k+1}}
\pm \Biggr(\prod_{r=1}^{k-1}L_r^{j_r}\Biggr)\sum_{c=\min\{j_k,j_{k+1}\}}^{\max\{j_{k},j_{k+1}\}-1} L_{k+1}^{j_{k}+j_{k+1}-c-1}L_{k}^{c}\Biggr)\Biggr(\prod_{p=k+2}^{n}L_p^{j_p}\Biggr) {\beta_{k+1,n-1}}\\
&=s_{k}\Biggr(\prod_{r=1}^{k-1}L_r^{j_r}\Biggr)L_{k}^{j_{k+1}}\Biggr(L_{k+1}^{j_{k}}\prod_{p=k+2}^{n}L_p^{j_p}\Biggr) {\beta_{k+1,n-1}}\\
&\qquad \pm \Biggr(\prod_{r=1}^{k-1}L_r^{j_r}\Biggr)\sum_{c=\min\{j_k,j_{k+1}\}}^{\max\{j_{k},j_{k+1}\}-1} L_{k+1}^{j_{k}+j_{k+1}-c-1}
 L_{k}^{c}\Biggr(\prod_{p=k+2}^{n}L_p^{j_p}\Biggr) \beta_{k+1,n-1}\\
&={\beta_{k,n-1}} \Biggr(\prod_{r=1}^{n}L_{\beta_{k,n-1}^{-1}(r)}^{j_r}\Biggr)
+\sum_{\substack{1\leq m\leq n-k\\k\leq i_1<\dots<i_m\leq n-1}} \beta^{i_1,\dots,i_m}_{k,n-1}f_{i_1,\dots,i_m}
\end{align*}
for some $f_{i_1,\dots,i_m}=\sum_{a=0}^{\ell-1}f_{i_1,\dots,i_m}^{(a)}L^a_n$
with $f_{i_1,\dots,i_m}^{(a)}\in R[L_1,\dots,L_{n-1}]$.
This completes the proof.
\end{proof}

\bigskip
\section{Proof of Theorem \ref{mainthm1}}

In this section, we shall give the proof of the first main result in this paper---Theorem \ref{mainthm1}. That is, we shall compute the explicit value of the standard symmetrizing form $\tau_{R}^{\rm{MM}}$ of the non-degenerate cyclotomic Hecke algebra $\HH_{\ell,n}$ on each Murphy basis element $\fm_{\s\t}$ of $\HH_{\ell,n}$. Throughout this section, unless otherwise stated, $R$ is an integral domain, $1\neq q\in R^\times$ and ${\mathbf Q}:=(Q_{1},Q_2,\dots, Q_{\ell})\in R^\ell$.

\subsection{Computing the explicit value of $\tau_R^{\rm{MM}}$ on $\fm_{\s\t}$}
In this subsection, we shall use the cyclotomic Mackey decomposition to compute the explicit value of $\tau_R^{\rm{MM}}$ on each Murphy basis element $\fm_{\s\t}$.

Let $\blam\in\mathscr{P}_{\ell,n}$ and $\t\in\te{Std}(\blam)$. Recall Definition \ref{keydefn} that
$1\leq n(\t)\leq n$ is the unique integer such that $\t^{-1}(n)=(\t^\blam)^{-1}(n(\t))$, and
for each integer $1\leq j\leq n$,  $$
\t^{-1}(j)=(r(\t,j),c(\t,j),l(\t,j))\in [\blam].$$
%For simplicity, we set $l(\blam,j):=l(\tlam,j)$.

We use ``$\leq $'' to denote the Bruhat order on $\Sym_n$.

\begin{prop}\label{nondegvemst}
Let $\blam\in\mathscr{P}_{\ell,n}$ and $\s,\t\in\te{Std}(\blam)$. Suppose that either $n(\s)=n$ or $n(\t)=n$. Then
$$\ve_n(\fm_{\s\t})=\begin{cases} \bigg(\prod\limits_{i=l(\t,n)+1}^{\ell}(-Q_i)\bigg)\fm_{\s_{\downarrow\leq (n-1)}\t_{\downarrow\leq (n-1)}}, &\text{if\ \ $n(\s)=n(\t),$}\\
0, &\text{if\ \ $n(\s)\neq n(\t).$}\end{cases}$$
\end{prop}

\begin{proof} We use $\ell_0$ to denote the largest integer such that $1\leq \ell_0\leq\ell$ and $\lam^{(\ell_0)}\neq\emptyset$. We prove the proposition by dividing it into three cases.

\smallskip\noindent
{\it Case 1.} Suppose $n(\s)=n=n(\t)$. Then we have $d(\s)=d(\s_{\downarrow\leq (n-1)})$,
$d(\t)=d(\t_{\downarrow\leq (n-1)})$, $l(\t,n)= \ell_0$, and $\te{Shape}(\s_{\downarrow\leq (n-1)})=\te{Shape}(\t_{\downarrow\leq (n-1)})=\blam_{\downarrow\leq (n-1)}.$
First we let $w\in\Sym_\blam$ with $\sigma_{n-1}\leq w$.
Recalling (\ref{decompositionofwlamr}), we have $w_{\blam,{\ell_0},r_{\ell_0}}=\beta_{b_{{\ell_0},r_{\ell_0}},n-1}\cdot w_{\blam,{\ell_0},r_{\ell_0}}^{(1)}$.
Thus using (\ref{decompositionofdt})-(\ref{decompositionofwlamr}) and Lemma \ref{nondeg-L-comm}
%and the equality $T_k(\mL_{k+1}-Q)=(\mL_k-Q)T_k+(q-1)\mL_{k+1}$
we deduce that
\begin{align*}
&T_{d(\s)}^*T_w \fu_\blam^+ T_{d(\t)}\\
 =\ &T_{d(\s_{\downarrow\leq (n-1)})}^* \Biggr(\prod_{i=1}^{\ell_0-1} T_{w_{i}}\Biggr)\Biggr(\prod_{j=1}^{r_{\ell_0}-1}T_{w_{\ell_0,j}}\Biggr)
T_{\beta_{b_{\ell_0,r_{\ell_0}},n-1}}T_{w_{\ell_0,r_{\ell_0}}^{(1)}} \Biggr(\prod_{s=2}^{\ell}\prod_{k=1}^{\fa_s}(\mL_k-Q_s)\Biggr) T_{d(\t_{\downarrow\leq (n-1)})}\\
 =\ &\pi_{\s,\t,w}^{(1,1)} T_{n-1}\pi^{(1,2)}_{\s,\t,w} + \rho^{(1)}_{\s,\t,w},
\end{align*}
where
$$\begin{aligned}
\pi_{\s,\t,w}^{(1,1)}&=
T_{d(\s_{\downarrow\leq (n-1)})}^* \Biggr(\prod\limits_{i=1}^{{\ell_0}-1} T_{w_{i}}\Biggr)\Biggr(\prod\limits_{j=1}^{r_{\ell_0}-1}T_{w_{{\ell_0},j}}\Biggr)
 T_{\beta_{b_{{\ell_0},r_{\ell_0}},n-2}}\Biggr(\prod\limits_{s=\ell_0+1}^{{\ell}}(\mL_{n-1}-Q_s)\Biggr)\in\HH_{{\ell},n-1},\\
\pi_{\s,\t,w}^{(1,2)}&=
T_{w_{{\ell_0},l_{\ell_0}}^{(1)}}
\Biggr(\prod\limits_{s=2}^{{\ell}}\prod\limits_{\substack{1\leq k\leq \fa_s\\k\neq n}}(\mL_k-Q_s)\Biggr) T_{d(\t_{\downarrow\leq (n-1)})}\in\HH_{{\ell},n-1},\end{aligned}$$ and  $$\begin{aligned}
\rho^{(1)}_{\s,\t,w}=
\begin{cases}\begin{matrix}
 T_{d(\s_{\downarrow\leq (n-1)})}^* \bigg(\prod\limits_{i=1}^{{\ell_0}-1} T_{w_{i}}\bigg)\bigg(\prod\limits_{j=1}^{r_{\ell_0}-1}T_{w_{\ell_0,j}}\bigg)
T_{\beta_{b_{{\ell_0},r_{\ell_0}},n-2}}\\
 \times \bigg(\sum\limits_{i=\ell_0+1}^{{\ell}}
\bigg( \prod\limits_{\ell_0+1\leq s<i}(\mL_{n-1}-Q_s)
\prod\limits_{i< t\leq {\ell}}(\mL_{n}-Q_t)\bigg)\bigg)\\
 \times (q-1)\mL_n T_{w_{{\ell_0},r_{\ell_0}}^{(1)}} \bigg(\prod\limits_{s=2}^{{\ell}}\prod\limits_{\substack{1\leq k\leq \fa_s\\k\neq n}}(\mL_k-Q_s)\bigg)
 T_{d(\t_{\downarrow\leq (n-1)})},\end{matrix} & \mbox{if\ \ } \ell_0<{\ell},\\
  0, & \mbox{if\ \ } \ell_0={\ell}.
\end{cases}
\end{aligned}$$
Then, noticing that
\begin{align}\label{easyobservation}
\forall\,\ell_0+1\leq i\leq \ell,\ \te{the degree of\ }\mL_n\te{ in }\prod\limits_{i< t\leq {\ell}}(\mL_{n}-Q_t)(q-1)\mL_n\te{\ lies in\ }[1,{\ell}-1],
\end{align}
we see that in this case \begin{equation}\label{Case1Ven}
\ve_n\bigl(T_{d(\s)}^*T_w \fu_\blam^+ T_{d(\t)}\bigr)=p_0^{(n)}(\rho^{(1)}_{\s,\t,w})=0.
\end{equation}
On the other hand, for each $w\in\Sym_\blam$ such that $\sigma_{n-1}\not\leq w$, using Lemma \ref{nondeg-L-comm} we have
\begin{align}\label{nondegcompu2}
 T_{d(\s)}^*T_w u_\blam^+ T_{d(\t)}
 &=T_{d(\s_{\downarrow\leq (n-1)})}^*T_w
\Biggl(\prod_{s=2}^{{\ell}}\prod_{\substack{1\leq k\leq \fa_s\\ k\neq n}}(\mL_k-Q_s)\Biggl)T_{d(\t_{\downarrow\leq (n-1)})}
\prod\limits_{i=\ell(\t,n)+1}^{{\ell}}(\mL_{n}-Q_i)\\
\notag &=T_{d(\s_{\downarrow\leq (n-1)})}^*T_w  \fu_{\blam_{\downarrow\leq (n-1)}}^+T_{d(\t_{\downarrow\leq (n-1)})}
\prod\limits_{i=\ell(\t,n)+1}^{{\ell}}(\mL_{n}-Q_i).
\end{align}
%The above element lies in $\sum_{j=0}^{\ell-1}\HH_{\ell,n-1}\mL_n^{j}$.
Therefore, combining (\ref{Case1Ven}) and (\ref{nondegcompu2}), we get that
\begin{align*}
\ve_n(\fm_{\s\t})&=p_0^{(n)}(\fm_{\s\t})
=\sum_{\substack{w\in\Sym_\blam\\ \sigma_{n-1}\not\leq w}}\Biggr(T_{d(\s_{\downarrow\leq (n-1)})}^*T_w  \fu_{\blam_{\downarrow\leq (n-1)}}^+T_{d(\t_{\downarrow\leq (n-1)})}\Biggr(\prod\limits_{i=\ell(\t,n)+1}^{{\ell}}(-Q_i)\Biggr)\Biggr)\\
\notag&=T_{d(\s_{\downarrow\leq (n-1)})}^*\Biggr(\sum_{w\in\Sym_{\blam_{\downarrow\leq (n-1)}}}T_w\Biggr)  \fu_{\blam_{\downarrow\leq (n-1)}}^+T_{d(\t_{\downarrow\leq (n-1)})}\Biggr(\prod\limits_{i=\ell(\t,n)+1}^{{\ell}}(-Q_i)\Biggr)\\
\notag &=T_{d(\s_{\downarrow\leq (n-1)})}^*x_{\blam_{\downarrow\leq (n-1)}}  \fu_{\blam_{\downarrow\leq (n-1)}}^+T_{d(\t_{\downarrow\leq (n-1)})}\Biggr(\prod\limits_{i=\ell(\t,n)+1}^{{\ell}}(-Q_i)\Biggr)\\
\notag&=\Biggr(\prod\limits_{i=\ell(\t,n)+1}^{{\ell}}(-Q_i)\Biggr) \fm_{\s_{\downarrow\leq (n-1)}\t_{\downarrow\leq (n-1)}}.
\end{align*}

\smallskip\noindent
{\it Case 2.} Suppose $n(\s)\neq n=n(\t)$. Then $d(\s)\neq d(\s_{\downarrow\leq (n-1)})$ and $d(\t)=d(\t_{\downarrow\leq (n-1)}).$
Then for each $w\in\Sym_\blam$ such that $\sigma_{n-1}\not\leq w,$
using (\ref{decompositionofdt}) and Lemma \ref{nondeg-L-comm} we deduce that
\begin{align}\label{nondegcompu3}
T_{d(\s)}^*T_w \fu_\blam^+ T_{d(\t)}
&=T_{d(\s_{\downarrow\leq (n-1)})}^*T_{\gamma_{n-1, n(\s)}} T_w
\Biggr(\prod_{s=2}^{{\ell}}\prod_{k=1}^{\fa_s}(\mL_k-Q_s)\Biggr)T_{d(\t_{\downarrow\leq (n-1)})}\\
\notag &=\pi_{\s,\t,w}^{(2,1)} T_{n-1}\pi^{(2,2)}_{\s,\t,w} + \rho^{(2)}_{\s,\t,w},
\end{align}
where
$$\begin{aligned}
\pi_{\s,\t,w}^{(2,1)}&=
T_{d(\s_{\downarrow\leq (n-1)})}^*
\Biggr(\prod\limits_{s=\ell_0+1}^{{\ell}}(\mL_{n-1}-Q_s)\Biggr)\in\HH_{{\ell},n-1},\\
\pi_{\s,\t,w}^{(2,2)}&=T_{\gamma_{n-2,n(\s)}}T_{w}
\Biggr(\prod\limits_{s=2}^{{\ell}}\prod\limits_{\substack{1\leq k\leq \fa_s\\k\neq n}}(\mL_k-Q_s)\Biggr) T_{d(\t_{\downarrow\leq (n-1)})}\in\HH_{{\ell},n-1},\end{aligned}$$ and  $$\begin{aligned}
\rho^{(2)}_{\s,\t,w}=
\begin{cases}\begin{matrix}
T_{d(\s_{\downarrow\leq (n-1)})}^* \bigg(\sum\limits_{i=\ell_0+1}^{{\ell}}
\bigg( \prod\limits_{\ell_0+1\leq s<i}(\mL_{n-1}-Q_s)
\prod\limits_{i< t\leq {\ell}}(\mL_{n}-Q_t)\bigg)\bigg) \\
\times (q-1)\mL_n T_{\gamma_{n-2,n(\s)}}T_w\bigg(\prod\limits_{s=2}^{{\ell}}\prod\limits_{\substack{1\leq k\leq \fa_s\\k\neq n}}(\mL_k-Q_s)\bigg)T_{d(\t_{\downarrow\leq (n-1)})},\end{matrix} & \mbox{if\ \ } \ell_0<{\ell},\\
  0, & \mbox{if\ \ } \ell_0={\ell}.
\end{cases}
\end{aligned}$$
As in Case 1, using (\ref{easyobservation}) we see that in this case, \begin{equation}\label{Case2Ven1}
\ve_n\bigl(T_{d(\s)}^*T_w \fu_\blam^+ T_{d(\t)}\bigr)=p_0^{(n)}(\rho^{(2)}_{\s,\t,w})=0.
\end{equation}
On the other hand, let $w\in\Sym_\blam$ such that $\sigma_{n-1} \leq w$.
Note that the assumption $\sigma_{n-1} \leq d(\s)$ implies that $n(\s)$ does not lie in the last row of $\s$, hence
$b_{{\ell_0},r_{\ell_0}}\geq a_{\ell_0,r_{\ell_0}-1}+1>n(\s)$. Using (\ref{nondegbraidrela1}), (\ref{decompositionofdt})-(\ref{decompositionofwlamr}), Lemma \ref{keytools} (1) and Lemma \ref{nondeg-L-comm} we can deduce that
\begin{align*}
T_{d(\s)}^*T_w\fu_\blam^+ T_{d(\t)}
\notag&= T_{d(\s_{\downarrow\leq (n-1)})}^*T_{\gamma_{n-1,n(\s)}}
\Biggr(\prod_{i=1}^{{\ell_0}-1} T_{w_{i}}\Biggr)\Biggr(\prod_{j=1}^{r_{\ell_0}-1}T_{w_{{\ell_0},j}}\Biggr)
T_{\beta_{b_{{\ell_0},r_{\ell_0}},n-1}}T_{w_{{\ell_0},r_{\ell_0}}^{(1)}} \\
\notag &\quad\ \times \Biggr(\prod_{s=2}^{{\ell}}\prod_{k=1}^{\fa_s}(\mL_k-Q_s)\Biggr)T_{d(\t_{\downarrow\leq (n-1)})}\\
\notag &= T_{d(\s_{\downarrow\leq (n-1)})}^*T_{\gamma_{n-1,n(\s)}} T_{\beta_{b_{{\ell_0},r_{\ell_0}},n-1}}
\Biggr(\prod_{i=1}^{{\ell_0}-1} T_{w_{i}}\Biggr)\Biggr(\prod_{j=1}^{r_{\ell_0}-1}T_{w_{{\ell_0},j}}\Biggr) \\
\notag&\quad\ \times T_{w_{{\ell_0},r_{\ell_0}}^{(1)}}
\Biggr(\prod_{s=2}^{{\ell}}\prod_{k=1}^{\fa_s}(\mL_k-Q_s)\Biggr)T_{d(\t_{\downarrow\leq (n-1)})}\\
%\notag=\ &T_{d(\s_{\downarrow\leq (n-1)})}^*T_{b_{w,{\ell_0},r_\ell}-1}\(T_{n-1}T_{n-2}\cdots %T_{k_{\s,n}}\)\(T_{b_{w,\ell,r_\ell}+1}T_{b_{w,\ell,r_\ell}+2}\cdots T_{n-1}\)\\
%\notag&\cdot \(\prod_{i=1}^{\ell-1} T_{w_{\blam,i}}\)\(\prod_{j=1}^{r_\ell-1}T_{w_{\blam,\ell,j}}\) T_{w_{\blam,\ell,r_\ell}^{(1)}}
%%\(\prod_{s=2}^{\ell}\prod_{k=1}^{\fa_s}(\mL_k-Q_s)\)T_{d(\t_{\downarrow\leq (n-1)})}\\
%\notag=\ &T_{d(\s_{\downarrow\leq (n-1)})}^*T_{b_{w,\ell,r_\ell}-1}T_{b_{w,\ell,r_\ell}}\(T_{n-1}T_{n-2}\cdots T_{k_{\s,n}}\)
%\(T_{b_{w,\ell,r_\ell}+2}T_{b_{w,\ell,r_\ell}+3}\cdots T_{n-1}\)\\
%%\notag&\cdot \(\prod_{i=1}^{\ell-1} T_{w_{\blam,i}}\)\(\prod_{j=1}^{r_\ell-1}T_{w_{\blam,\ell,j}}\) T_{w_{\blam,\ell,r_\ell}^{(1)}}
%\(\prod_{s=2}^{\ell}\prod_{k=1}^{\fa_s}(\mL_k-Q_s)\)T_{d(\t_{\downarrow\leq (n-1)})}\\
%%\notag &\ \ \vdots\\
%\notag=\ &T_{d(\s_{\downarrow\leq (n-1)})}^*\(T_{b_{{\ell_0},r_{\ell_0}}-1}T_{b_{{\ell_0},r_{\ell_0}}}\cdots T_{n-2}\)T_{\gamma_{n-1,n(\s)}}\\
%\notag&\times \(\prod_{i=1}^{{\ell_0}-1} T_{w_{i}}\)\(\prod_{j=1}^{r_{\ell_0}-1}T_{w_{{\ell_0},j}}\) T_{w_{{\ell_0},r_{\ell_0}}^{(1)}}
%\(\prod_{s=2}^{{\ell}}\prod_{k=1}^{\fa_s}(\mL_k-Q_s)\)T_{d(\t_{\downarrow\leq (n-1)})}\\
\notag&= T_{d(\s_{\downarrow\leq (n-1)})}^*T_{\beta_{b_{{\ell_0},r_{\ell_0}}-1,n-2}}T_{n-1}T_{\gamma_{n-2,n(\s)}}
\Biggr(\prod_{i=1}^{{\ell_0}-1} T_{w_{i}}\Biggr)\Biggr(\prod_{j=1}^{r_{\ell_0}-1}T_{w_{{\ell_0},j}}\Biggr) \\
\notag&\quad\ \times T_{w_{{\ell_0},r_{\ell_0}}^{(1)}}
\Biggr(\prod_{s=2}^{{\ell}}\prod_{k=1}^{\fa_s}(\mL_k-Q_s)\Biggr)T_{d(\t_{\downarrow\leq (n-1)})}\\
\notag &=\pi_{\s,\t,w}^{(3,1)} T_{n-1}\pi^{(3,2)}_{\s,\t,w} + \rho^{(3)}_{\s,\t,w},
\end{align*}
where
$$\begin{aligned}
\pi_{\s,\t,w}^{(3,1)}&=
T_{d(\s_{\downarrow\leq (n-1)})}^*T_{\beta_{b_{{\ell_0},r_{\ell_0}}-1,n-2}}
\Biggr(\prod\limits_{s=l(\blam,n)+1}^{{\ell}}(\mL_{n-1}-Q_s)\Biggr)\in\HH_{{\ell},n-1},\\
\pi_{\s,\t,w}^{(3,2)}&=T_{\gamma_{n-2,n(\s)}}
\Biggr(\prod_{i=1}^{{\ell_0}-1} T_{w_{i}}\Biggr)\Biggr(\prod_{j=1}^{r_{\ell_0}-1}T_{w_{{\ell_0},j}}\Biggr) T_{w_{{\ell_0},r_{\ell_0}}^{(1)}}
\Biggr(\prod\limits_{s=2}^{{\ell}}\prod\limits_{\substack{1\leq k\leq \fa_s\\k\neq n}}(\mL_k-Q_s)\Biggr)\\
&\qquad \times T_{d(\t_{\downarrow\leq (n-1)})}\in\HH_{{\ell},n-1},\end{aligned}$$ and  $$\begin{aligned}
\rho^{(3)}_{\s,\t,w}=
\begin{cases}\begin{matrix}
 T_{d(\s_{\downarrow\leq (n-1)})}^* T_{\beta_{b_{{\ell_0},r_{\ell_0}}-1,n-2}}
 \bigg(\sum\limits_{i=\ell_0+1}^{{\ell}}\bigg( \prod\limits_{\ell_0+1\leq s<i}(\mL_{n-1}-Q_s)\\
 \times \prod\limits_{i< t\leq {\ell}}(\mL_{n}-Q_t)\bigg)\bigg)
 (q-1)\mL_n T_{\gamma_{n-2,n(\s)}}\bigg(\prod\limits_{i=1}^{{\ell_0}-1} T_{w_{i}}\bigg) \\
 \times\bigg(\prod\limits_{j=1}^{r_{\ell_0}-1}T_{w_{{\ell_0},j}}\bigg) T_{w_{{\ell_0},r_{\ell_0}}^{(1)}} \bigg(\prod\limits_{s=2}^{{\ell}}\prod\limits_{\substack{1\leq k\leq \fa_s\\k\neq n}}(\mL_k-Q_s)\bigg)T_{d(\t_{\downarrow\leq (n-1)})},\end{matrix} & \mbox{if\ \ }\ell_0<{\ell},\\
  0, & \mbox{if\ \ } \ell_0={\ell}.
\end{cases}
\end{aligned}$$
Then still using (\ref{easyobservation}) we see that in this case, \begin{equation}\label{Case2Ven2}
\ve_n\bigl(T_{d(\s)}^*T_w \fu_\blam^+ T_{d(\t)}\bigr)=p_0^{(n)}(\rho^{(3)}_{\s,\t,w})=0.
\end{equation}
Therefore, combining (\ref{Case2Ven1}) and (\ref{Case2Ven2}), we get
$\ve_n(\fm_{\s\t})=0.$

\smallskip\noindent
{\it Case 3.} Suppose that $n(\s)=n\neq n(\t)$. Then it follows from Lemma \ref{veAst}, (\ref{nondegmst*}) and the result of Case 2 that
$$\ve_n(\fm_{\s\t})=\ve_n(\fm_{\s\t}^*)=\ve_n(\fm_{\t\s})=0.$$
This completes the proof of the proposition.
\end{proof}

\begin{prop}\label{nondegvemst2}
Let $\blam\in\mathscr{P}_{\ell,n}$ and $\s,\t\in\te{Std}(\blam)$. Suppose that $n(\s)\neq n\neq n(\t)$. Then
$$\ve_n(\fm_{\s\t})=\begin{cases} q^{n-n(\t)}\bigg(\prod\limits_{i=l(\t,n)+1}^{\ell}(-Q_i)\bigg)\fm_{\s_{\downarrow\leq (n-1)}\t_{\downarrow\leq (n-1)}}, &\text{if\ \,$n(\s)=n(\t),$}\\
0, &\text{if\ \,$n(\s)\neq n(\t).$}\end{cases}$$
\end{prop}

%$d(\s)\neq d(\s_{\downarrow\leq (n-1)})$ and $d(\t)\neq d(\t_{\downarrow\leq (n-1)})$.

\begin{proof}  We use $\ell_0$ to denote the largest integer such that $1\leq \ell_0\leq\ell$ and $\lam^{(\ell_0)}\neq\emptyset$.
By assumption, $\sigma_{n-1} \leq d(\s)$ and $\sigma_{n-1} \leq d(\t)$. In view of Lemma \ref{veAst} and (\ref{nondegmst*}), we may assume that $n(\s)\leq n(\t)$.

For each $w\in\Sym_\blam$, using
(\ref{decompositionofdt})-(\ref{decompositionofwlamr}) and Lemma \ref{keytools} (1) we deduce that
\begin{align}\label{nondegcompu5}
&T_{d(\s)}^*T_w\fu_\blam^+ T_{d(\t)}\\
\notag =\ &T_{d(\s_{\downarrow\leq (n-1)})}^*T_{\gamma_{n-1,n(\s)}}\Biggr(\prod_{j=1}^{\ell_0} T_{w_{j}}\Biggr)
\Biggr(\prod_{s=2}^{\ell}\prod_{i=1}^{\fa_s}(\mL_i-Q_s)\Biggr)
T_{\beta_{n(\t),n-1}} T_{d(\t_{\downarrow\leq (n-1)})}\\
\notag =\ &T_{d(\s_{\downarrow\leq (n-1)})}^*\Biggr(\prod_{a=1}^{l(\s,n)-1} T_{w_{a}} \Biggr)
\Biggr(\prod_{b=1}^{r(\s,n)-1} T_{w_{l(\s,n),b}}\Biggr)
\Biggr(\prod_{c=r(\s,n)+1}^{r_{l(\s,n)}} T_{(w_{l(\s,n),c})_{\downarrow}} \Biggr)\\
\notag &\times \Biggr(\prod_{d=l(\s,n)+1}^{\ell_0} T_{(w_{d})_{\downarrow}} \Biggr)
 T_{\gamma_{n-1,n(\s)}}T_{w_{l(\s,n),r(\s,n)}}  \Biggr(\prod_{s=2}^{\ell}\prod_{i=1}^{\fa_s}(\mL_i-Q_s)\Biggr)T_{\beta_{n(\t),n-1}}T_{d(\t_{\downarrow\leq (n-1)})}.
\end{align}
We want to compute $\ve_n(T_{d(\s)}^*T_w\fu_\blam^+ T_{d(\t)})$. Since both the elements $$
T_{d(\s_{\downarrow\leq (n-1)})}^*\Biggr(\prod_{a=1}^{l(\s,n)-1} T_{w_{a}} \Biggr)
\Biggr(\prod_{b=1}^{r(\s,n)-1} T_{w_{l(\s,n),b}}\Biggr)
\Biggr(\prod_{c=r(\s,n)+1}^{r_{l(\s,n)}} T_{(w_{l(\s,n),c})_{\downarrow}} \Biggr)
 \Biggr(\prod_{d=l(\s,n)+1}^{\ell_0} T_{(w_{d})_{\downarrow}} \Biggr)$$
and $T_{d(\t_{\downarrow\leq (n-1)})}$ lie in $\HH_{\ell,n-1}$, in view of Lemma \ref{veAst}, it suffices to consider the cyclotomic Mackey decomposition of
\begin{equation*}
T_{\gamma_{n-1,n(\s)}}T_{w_{l(\s,n),r(\s,n)}}  \Biggr(\prod_{s=2}^{\ell}\prod_{i=1}^{\fa_s}(\mL_i-Q_s)\Biggr)T_{\beta_{n(\t),n-1}}
\end{equation*}
and its value under $\ve_n.$

Denote that $\bnu:=\te{Shape}(\t_{\downarrow\leq (n-1)})$.
By definition, $$
\begin{aligned}
\fu_{\bnu}^+=\Biggr(\prod_{s=2}^{l(\t,n)}\prod_{i=1}^{\mathfrak{a}_s}(\mL_i-Q_s)\Biggr)
\Biggr(\prod_{k=l(\t,n)+1}^{\ell}\prod_{j=1}^{\mathfrak{a}_s-1}(\mL_j-Q_k)\Biggr)
=\prod_{s=2}^{\ell}\prod_{\substack{1\leq i\leq\fa_s\\ i\neq n(\t)}}\(\mL_{\beta_{n(\t),n-1}^{-1}(i)}-Q_s\).
\end{aligned}$$
Thus applying Lemma \ref{JucyMurphyCommutators}, we get that
\begin{align*}
& \Biggr(\prod_{s=2}^{\ell}\prod_{i=1}^{\fa_s}(\mL_i-Q_s)\Biggr)T_{\beta_{n(\t),n-1}}\\
=\ &T_{\beta_{n(\t),n-1}}\prod_{s=2}^{\ell}\prod_{i=1}^{\fa_s}(\mL_{\beta_{n(\t),n-1}^{-1}(i)}-Q_s)
+\sum_{\substack{1\leq m\leq n-k\\k\leq i_1<i_2<\dots<i_m\leq n-1}} T_{\beta^{i_1,\dots,i_m}_{k,n-1}}f_{i_1,\dots,i_m}\\
=\ &T_{\beta_{n(\t),n-1}}\prod_{s=2}^{\ell}\prod_{\substack{1\leq i\leq\fa_s\\ i\neq n(\t)}}(\mL_{\beta_{n(\t),n-1}^{-1}(i)}-Q_s)
\Biggr(\prod_{j=l(\t,n)+1}^{\ell}(\mL_{n}-Q_j)\Biggr)\\
& +\sum_{\substack{1\leq m\leq n-k\\k\leq i_1<i_2<\dots<i_m\leq n-1}} T_{\beta^{i_1,\dots,i_m}_{k,n-1}}f_{i_1,\dots,i_m}\\
=\ &T_{\beta_{n(\t),n-1}}\fu_\bnu^+\Biggr(\prod_{i=l(\t,n)+1}^{\ell}(\mL_{n}-Q_i)\Biggr)
+\sum_{\substack{1\leq m\leq n-k\\k\leq i_1<i_2<\dots<i_m\leq n-1}} T_{\beta^{i_1,\dots,i_m}_{k,n-1}}f_{i_1,\dots,i_m},
\end{align*}
where $f_{i_1,\dots,i_m}=\sum_{a=0}^{\ell-1}f_{i_1,\dots,i_m}^{(a)}\mL^a_n$
for some $f_{i_1,\dots,i_m}^{(a)}\in R[\mL_1,\dots,\mL_{n-1}]$.

Suppose that $\sigma_{n(\s)-1} \leq w_{l(\s,n),r(\s,n)}$. Note that Lemma \ref{nondeg-L-comm} implies that
\begin{align}\label{propertyTLcomm}
T_{n-1}\mL_n^k\in \HH_{\ell,n-1}T_{n-1}\HH_{\ell,n-1}+\sum_{j=1}^{\ell-1}\HH_{\ell,n-1}\mL_n^{j},\ \ \forall\,0\leq k\leq \ell-1.
\end{align}
Noticing that $c_{l(\s,n),r(\s,n)}\leq n(\s)-1< n(\t),$
and using (\ref{decompositionofwlamr}), (\ref{nondegbraidrela1}),  Lemma 2.17 (1) and (\ref{propertyTLcomm}),
we deduce that: for each $1\leq m\leq n-n(\t)$ and $n(\t)\leq i_1<\dots<i_m\leq n-1$,
\begin{align}\label{nondegcompu7}
&\quad\,T_{\gamma_{n-1,n(\s)}}T_{w_{l(\s,n),r(\s,n)}}
T_{\beta^{i_1,\dots,i_m}_{n(\t),n-1}}f_{i_1,\dots,i_m}\\
\notag & =T_{\gamma_{n-1,n(\s)}}T_{w_{l(\s,n),r(\s,n)}^{(2)}}T_{\gamma_{n(\s)-1,c_{l(\s,n),r(\s,n)}}}
T_{\beta^{i_1,\dots,i_m}_{n(\t),n-1}}f_{i_1,\dots,i_m}\\
\notag & =T_{w_{l(\s,n),r(\s,n)}^{(2)}}
T_{\gamma_{n-1,c_{l(\s,n),r(\s,n)}}}T_{\beta^{i_1,\dots,i_m}_{n(\t),n-1}}
f_{i_1,\dots,i_m} \\
\notag & =T_{w_{l(\s,n),r(\s,n)}^{(2)}}
T_{(\beta^{i_1,\dots,i_m}_{n(\t),n-1})_{\downarrow}}T_{\gamma_{n-1,c_{l(\s,n),r(\s,n)}}}
f_{i_1,\dots,i_m}\\
\notag &= \mu_{T_{n-1}}(\pi^{(4,i_1,\dots,i_m)}_{\s,\t,w})+\sum_{j=1}^{\ell-1}g_j^{i_1,\dots,i_m} \mL_n^j
\end{align}
for some $\pi^{(4,i_1,\dots,i_m)}_{\s,\t,w}\in \HH_{\ell,n-1}\ot_{\HH_{\ell,n-2}}\HH_{\ell,n-1}$
and $g_j^{i_1,\dots,i_m}\in \HH_{\ell,n-1},\ 1\leq j\leq \ell-1$. Similarly,
\begin{align}\label{nondegcompu8}
&\quad\, T_{\gamma_{n-1,n(\s)}}T_{w_{l(\s,n),r(\s,n)}}T_{\beta_{n(\t),n-1}}\fu_\bnu^+ \Biggr(\prod_{i=l(\t,n)+1}^{\ell}(\mL_{n}-Q_i)\Biggr) \\
\notag &=T_{\gamma_{n-1,n(\s)}}T_{w_{l(\s,n),r(\s,n)}^{(2)}}T_{\gamma_{n(\s)-1,c_{l(\s,n),r(\s,n)}}}T_{\beta_{n(\t),n-1}}\fu_\bnu^+ \Biggr(\prod_{i=l(\t,n)+1}^{\ell}(\mL_{n}-Q_i)\Biggr)\\
 \notag &= T_{w_{l(\s,n),r(\s,n)}^{(2)}}
 T_{\gamma_{n-1,c_{l(\s,n),r(\s,n)}}}T_{\beta_{n(\t),n-1}}\fu_\bnu^+ \Biggr(\prod_{i=l(\t,n)+1}^{\ell}(\mL_{n}-Q_i)\Biggr)\\
\notag &= T_{w_{l(\s,n),r(\s,n)}^{(2)}}T_{\beta_{n(\t)-1,n-2}}
 T_{\gamma_{n-1,c_{l(\s,n),r(\s,n)}}}\fu_\bnu^+ \Biggr(\prod_{i=l(\t,n)+1}^{\ell}(\mL_{n}-Q_i)\Biggr)\\
\notag &=T_{w_{l(\s,n),r(\s,n)}^{(2)}}T_{\beta_{n(\t)-1,n-2}}
T_{n-1}\Biggr(\prod_{i=l(\t,n)+1}^{\ell}(\mL_{n}-Q_i)\Biggr)T_{\gamma_{n-2,c_{l(\s,n),r(\s,n)}}}\fu_\bnu^+\\
\notag &= \mu_{T_{n-1}}(\pi^{(5)}_{\s,\t,w})+\sum_{j=1}^{\ell-1}g_j \mL_n^j
\end{align}
for some $\pi^{(5)}_{\s,\t,w}\in \HH_{\ell,n-1}\ot_{\HH_{\ell,n-2}}\HH_{\ell,n-1}$
and $g_j \in \HH_{\ell,n-1},\ 1\leq j\leq \ell-1$.
Thus for (\ref{nondegcompu7}) and (\ref{nondegcompu8}), their values under $\ve_n$ are both zero by the cyclotomic Mackey decomposition (\ref{nondegMadecomequation}).
Therefore, if $\sigma_{n(\s)-1} \leq w_{l(\s,n),r(\s,n)}$, then
$\ve_n(T_{d(\s)}^*T_w\fu_\blam^+ T_{d(\t)})=0$.

Now assume that $\sigma_{n(\s)-1}\not\leq w_{l(\s,n),r(\s,n)}$, which implies that $w_{l(\s,n),r(\s,n)}$ commutes with $\gamma_{n-1,n(\s)}$.
 Using the assumption $n(\s)\leq n(\t)$,
 the relation (\ref{nondegbraidrela1}), (1) and (2) of Lemma \ref{keytools} and (\ref{propertyTLcomm}),
we deduce that: for each $1\leq m\leq n-n(\t)$ and $n(\t)\leq i_1<\dots<i_m\leq n-1$,
\begin{align}\label{nondegcompu9}
&T_{\gamma_{n-1,n(\s)}}T_{w_{l(\s,n),r(\s,n)}}
T_{\beta^{i_1,\dots,i_m}_{n(\t),n-1}}f_{i_1,\dots,i_m}T_{d(\t_{\downarrow\leq (n-1)})}\\
\notag =\ &T_{w_{l(\s,n),r(\s,n)}}
 T_{\gamma_{n-1,n(\s)}}T_{\beta^{i_1,\dots,i_m}_{n(\t),n-1}}f_{i_1,\dots,i_m}T_{d(\t_{\downarrow\leq (n-1)})}\\
\notag=\ &\mu_{T_{n-1}}(\pi^{(6,i_1,\dots,i_m)}_{\s,\t,w})+\sum_{j=1}^{\ell-1}h_j^{i_1,\dots,i_m} \mL_n^j
\end{align}
for some $\pi^{(6,i_1,\dots,i_m)}_{\s,\t,w}\in \HH_{\ell,n-1}\ot_{\HH_{\ell,n-2}}\HH_{\ell,n-1}$
and $h_j^{i_1,\dots,i_m}\in \HH_{\ell,n-1},\ 1\leq j\leq \ell-1$. So, once again, the value of the term (\ref{nondegcompu9}) under $\ve_n$ is zero by the cyclotomic Mackey decomposition (\ref{nondegMadecomequation}).

Therefore, it remains to consider the following term under the assumption that $\sigma_{n(\s)-1}\not\leq w_{l(\s,n),r(\s,n)}$. Noticing that in this case $w_{l(\s,n),r(\s,n)}$ commutes with $\gamma_{n-1,n(\s)}$, we get that
 \begin{align}\label{wlsnrsncommnondeg}
  &  T_{\gamma_{n-1,n(\s)}}T_{w_{l(\s,n),r(\s,n)}}T_{\beta_{n(\t),n-1}}
 \fu_\bnu^+ \Biggr(\prod_{i=l(\t,n)+1}^{\ell}(\mL_{n}-Q_i)\Biggr)T_{d(\t_{\downarrow\leq (n-1)})} \\
\notag=\ &T_{w_{l(\s,n),r(\s,n)}}
  T_{\gamma_{n-1,n(\s)}}T_{\beta_{n(\t),n-1}}\Biggr(\prod_{i=l(\t,n)+1}^{\ell}(\mL_{n}-Q_i)\Biggr)\fu_\bnu^+ T_{d(\t_{\downarrow\leq (n-1)})} .
\end{align}

If $n(\s)< n(\t)$, then using Lemma \ref{keytools} (1) and (\ref{propertyTLcomm})  we have $$\begin{aligned}
&T_{w_{l(\s,n),r(\s,n)}}T_{\gamma_{n-1,n(\s)}}T_{\beta_{n(\t),n-1}}\Biggr(\prod_{i=l(\t,n)+1}^{\ell}(\mL_{n}-Q_i)\Biggr)\\
=\ &T_{w_{l(\s,n),r(\s,n)}}T_{\beta_{n(\t)-1,n-2}}T_{n-1}\Biggr(\prod_{i=l(\t,n)+1}^{\ell}(\mL_{n}-Q_i)\Biggr)T_{\gamma_{n-2,n(\s)}}\\
=\ &\mu_{T_{n-1}}(\pi_{\s\t})+\sum_{j=1}^{\ell-1}h_j\mL_n^j
\end{aligned}
$$
for some $\pi_{\s\t}\in \HH_{\ell,n-1}\ot_{\HH_{\ell,n-2}}\HH_{\ell,n-1}$ and $h_j\in \HH_{\ell,n-1},\ 1\leq j\leq \ell-1$.
So, in this case, the value of (\ref{wlsnrsncommnondeg}) under $\ve_n$ is zero by the cyclotomic Mackey decomposition (\ref{nondegMadecomequation}).

If $n(\s)=n(\t)$, then by Lemma \ref{keytools} (3), we can find $\pi'_{\s\t} \in \mathscr{H}_{\ell,n-1} \otimes_{\mathscr{H}_{\ell,n-2}} \mathscr{H}_{\ell,n-1}$, such that
\begin{align}\label{keydecompositionnondeg}
T_{\gamma_{n-1,n(\s)}}T_{\beta_{n(\t),n-1}}=\mu_{T_{n-1}}(\pi'_{\s\t})+q^{n-n(\t)}.
\end{align}
Thus it follows from (\ref{nondegcompu5}), (\ref{wlsnrsncommnondeg}), (\ref{keydecompositionnondeg}) and (\ref{propertyTLcomm}) that in this case   $$\begin{aligned}
\ve_n(T_{d(\s)}^*T_w\fu_\blam^+ T_{d(\t)})
&=q^{n-n(\t)}\Biggr(\prod\limits_{i=l(\t,n)+1}^{\ell}(-Q_i)\Biggr)T_{d(\s_{\downarrow\leq (n-1)})}^*\Biggr(\prod_{a=1}^{l(\s,n)-1} T_{w_{a}} \Biggr)
\Biggr(\prod_{b=1}^{r(\s,n)-1} T_{w_{l(\s,n),b}}\Biggr)\\
&\qquad  \times \Biggr(\prod_{c=r(\s,n)+1}^{r_{l(\s,n)}} T_{(w_{l(\s,n),c})_{\downarrow}} \Biggr)\Biggr(\prod_{d=l(\s,n)+1}^{\ell_0}T_{(w_{d})_{\downarrow}} \Biggr)\fu_\bnu^+ T_{d(\t_{\downarrow\leq (n-1)})}.
\end{aligned}$$

Note that if $n(\s)=n(\t)$, then $\te{Shape}(\s_{\downarrow\leq (n-1)})=\te{Shape}(\t_{\downarrow\leq (n-1)})=\bnu. $
Thus in this case we deduce that
\begin{align}\label{Sbnu}
\Biggr\{\begin{matrix}
\(\prod_{a=1}^{l(\s,n)-1} w_{a} \)
\(\prod_{b=1}^{r(\s,n)} w_{l(\s,n),b}\)\\
  \times\(\prod_{c=r(\s,n)+1}^{r_{l(\s,n)}} (w_{l(\s,n),c})_{\downarrow} \)
  \(\prod_{d=l(\s,n)+1}^{\ell_0} (w_{d})_{\downarrow} \)
\end{matrix}
\,\Biggr|\, \begin{matrix}w\in\Sym_\blam,\\
 \sigma_{n(\s)-1}\not\leq w_{l(\s,n),r(\s,n)}
 \end{matrix}\Biggr\}
= \Sym_\bnu,\end{align}
hence
\begin{align}\label{nondegcompu11}
&\sum_{\substack{w\in\Sym_\blam\\ \sigma_{n(\s)-1}\not\leq w_{l(\s,n),r(\s,n)}}}\Biggr(\prod\limits_{a=1}^{l(\s,n)-1} T_{w_{a}} \Biggr)
\Biggr(\prod\limits_{b=1}^{r(\s,n)} T_{w_{l(\s,n),b}}\Biggr)
\Biggr(\prod\limits_{c=r(\s,n)+1}^{r_{l(\s,n)}} T_{(w_{l(\s,n),c})_{\downarrow}} \Biggr)\\
 \notag  &\qquad\qquad\qquad  \times \Biggr(\prod\limits_{d=l(\s,n)+1}^{\ell_0} T_{(w_{d})_{\downarrow}} \Biggr)
=\sum_{u\in \Sym_\bnu}T_u=x_\bnu.
\end{align}

Therefore, summarizing all the discussion above, we get that: if $n(\s)\neq n(\t),$ then $\ve_n(\fm_{\s\t})=0;$ if $n(\s)=n(\t)$, then
\begin{align*}
\ve_n(\fm_{\s\t})
&= \sum_{\substack{w\in\Sym_\blam\\ \sigma_{n(\s)-1}\not\leq w_{l(\s,n),r(\s,n)}}}
q^{n-n(\t)}\Biggr(\prod\limits_{i=l(\t,n)+1}^{\ell}(-Q_i)\Biggr)T_{d(\s_{\downarrow\leq (n-1)})}^*
\Biggr(\prod\limits_{a=1}^{l(\s,n)-1} T_{w_{a}} \Biggr)\\
&\qquad \qquad  \times\Biggr(\prod\limits_{b=1}^{r(\s,n)} T_{w_{l(\s,n),b}}\Biggr) \Biggr(\prod\limits_{c=r(\s,n)+1}^{r_{l(\s,n)}} T_{(w_{l(\s,n),c})_{\downarrow}} \Biggr)
  \Biggr(\prod\limits_{d=l(\s,n)+1}^{\ell_0} T_{(w_{d})_{\downarrow}} \Biggr) \fu_\bnu^+ T_{d(\t_{\downarrow\leq (n-1)})}\\
&= q^{n-n(\t)}\Biggr(\prod\limits_{i=l(\t,n)+1}^{\ell}(-Q_i)\Biggr) T_{d(\s_{\downarrow\leq (n-1)})}^*
x_\bnu \fu^+_\bnu T_{d(\t_{\downarrow\leq (n-1)})}\\
&= q^{n-n(\t)}\Biggr(\prod\limits_{i=l(\t,n)+1}^{\ell}(-Q_i)\Biggr)\fm_{\s_{\downarrow\leq (n-1)}\t_{\downarrow\leq (n-1)}}.
\end{align*}
This completes the proof of the proposition. \end{proof}

\begin{prop}\label{MainProp1}
Let $\blam\in\mathscr{P}_{\ell,n}$ and $\s,\t\in\te{Std}(\blam).$ Then
$$\ve_n(\fm_{\s\t})=\begin{cases}
q^{n-n(\t)}\bigg(\prod\limits_{i=l(\t,n)+1}^{\ell}(-Q_i)\bigg)\fm_{\s_{\downarrow\leq (n-1)}\t_{\downarrow\leq (n-1)}},
& \text{if\ \ $n(\s)=n(\t),$}\\
0, & \text{if\ \ $n(\s)\neq n(\t)$.}
\end{cases} $$
\end{prop}
\begin{proof}
This follows from Propositions \ref{nondegvemst} and \ref{nondegvemst2}.
\end{proof}

For any $\s,\t\in\Std(\blam)$, it is clear that $\s\neq \t$ if and only if there is an integer $1\leq k\leq n$
such that $k(\s_{\downarrow\leq k})\neq k(\t_{\downarrow\leq k})$. As a consequence,
we get the proof of Theorem \ref{mainthm1}.\medskip

\noindent
{\bf Proof of Theorem \ref{mainthm1}:} This follows from Lemma \ref{nondegtauoperation} and Proposition \ref{MainProp1}.
\hfill\qed

\subsection{An application}
In the rest of this section, we shall give an application of Proposition \ref{MainProp1} and Theorem \ref{mainthm1}. From now on until the end of this section, we shall assume that $R=K$ is a field and $\HH_{\ell,n}$ is semisimple over $K$.

Let $\blam\in\P_{\ell,n}$ and $\t\in\Std(\blam)$. For each $1\leq k\leq n$, we define the residue of $k$ in $\t$ to be
$$\res_{\t}(k)=q^{c-b}Q_{l}\te{\quad if\quad} \t^{-1}(k)=(b,c,l).$$
For each $1\leq k\leq n$, we define $C(k):=\left\{\res_\t(k)\mid \t \in \Std(\blam),\ \blam\in \P_{\ell, n}\right\}$.

\begin{dfn}\text{\rm (\cite{Mur92}, \cite[Definition 2.4]{Ma})}\label{dfn:Ft}
For any $\s,\t\in\Std(\blam)$ and $\blam\in\P_{\ell,n},$ we define
 $$
{F}_{\t}:=\prod\limits^n\limits_{k=1}\prod\limits_{\substack{c\in C(k)\\c\neq\res_{\t}(k)}}\frac{\mL_k-c}{\res_{\t}(k)-c}\in \HH_{\ell,n},
\quad \mf_{\s\t}:={F}_\s \fm_{\s\t}{F}_\t\in\HH_{\ell,n}.
$$
\end{dfn}
By \cite{Ma}, the set $\{\mf_{\s\t}\,|\, \s,\t\in\Std(\blam),\ \blam\in\P_{\ell,n}\}$ gives a $K$-basis of $\HH_{\ell,n}$, and is called the seminormal basis of $\HH_{\ell,n}$.

Given two nodes $x=(i,j,k)$ and $y=(a,b,c)$, we write $y\prec x$ if either $c>k$ or $c=k$ and $b<j$. Let $\t\in\Std(\blam)$, $\gamma=\t^{-1}(i)$ and $1\leq i\leq n$. Let $\mathscr{A}_{\t}(i)$ be the set of addable nodes for $\Shape(\t_{\downarrow\leq i})$ which are strictly lesser than $\gamma$ with respect to $\prec$. Let $\mathscr{R}_{\t}(i)$ be the set of removable nodes for $\Shape(\t_{\downarrow\leq(i-1)})$ which are strictly lesser than $\gamma$ with respect to $\prec$.

Let $\blam\in\P_{\ell,n}$ and $\t\in\te{Std}(\blam).$
As in \cite[(2.8)]{Ma}, we define
$$\gamma_\t:=q^{\ell(d(\t))+\alpha(\blam)}\prod_{i=1}^{n}\frac{\prod\limits_{\gamma\in \mathscr{A}_{\ft}(i)}\(\res_\ft(i)-\res(\gamma)\)}
{\prod\limits_{\alpha\in \mathscr{R}_{\ft}(i)}\(\res_\ft(i)-\res(\alpha)\)}.$$

\begin{lem}\text{\rm (\cite[2.5, 2.6, 2.11]{Ma})}\label{fsemi}
\begin{itemize}
  \item[(1)] For any $\s,\t\in\Std(\blam)$ and $\u,\v\in\Std(\bmu),$ where $\blam,\bmu\in\P_{\ell,n},$ we have $$
\mathfrak{f}_{\s\t}\mathfrak{f}_{\u\v}=\delta_{\t\u}\gamma_\t \mathfrak{f}_{\s\v}.
$$
  \item[(2)] Let $\blam\in\P_{\ell,n}$ and $\s,\t\in\Std(\blam)$. Then for each $1\leq k\leq n,$ $$
\mL_k \mathfrak{f}_{\s\t}=\res_\s(k)\mathfrak{f}_{\s\t},\quad  \mathfrak{f}_{\s\t}\mL_k=\res_\t(k)\mathfrak{f}_{\s\t}. $$
Moreover, $$
K\mff_{\s\t}=\bigl\{h\in\HH_{\ell,n} \bigm|\mL_kh=\res_\s(k)h,\ h\mL_k=\res_\t(k)h,\ \forall\,1\leq k\leq n\bigr\}.
$$
  \item[(3)] For any $\s,\t\in\Std(\P_{\ell,n}),$ we have $\s=\t$ if and only if $\res_\s(k)=\res_\t(k),\ \forall\,1\leq k\leq n$.
\end{itemize}
%In particular, $F_\s:=\mathfrak{f}_{\s\s}/\gamma_\s$ is a primitive idempotent, and $$
%\tau_n^{\K}(\mathfrak{f}_{\s\t})=\delta_{\s\t}\frac{\gamma_{\s}}{s_\blam(q;\bQ)},$$
%where $s_\blam(q,\bQ)$ is the Schur element of $\HH_{\ell,n}^{\K}$ associated to $\blam$.
\end{lem}

\begin{prop}\te{(}\cite[Proposition 2.6]{Ma}, \cite[Theorem 2.9]{HM12}\te{)}\label{nondegbasestranslationrela}
Let $\blam\in\P_{\ell,n}$ and $\s,\t\in\te{Std}(\blam).$ Then we have
$$\ff_{\s\t}=\fm_{\s\t}+\sum_{\substack{(\u,\v)\in\te{Std}^2(\P_{\ell,n})\\(\u,\v)\rhd (\s,\t)}}a_{\u\v}^{\s\t}\fm_{\u\v},\quad
\fm_{\s\t}=\ff_{\s\t}+\sum_{\substack{(\u,\v)\in\te{Std}^2(\P_{\ell,n})\\(\u,\v)\rhd (\s,\t)}}b_{\u\v}^{\s\t}\ff_{\u\v}, $$
where $a_{\u\v}^{\s\t},b_{\u\v}^{\s\t}\in K$ for each $(\u,\v)\in\te{Std}^2(\P_{\ell,n})$ with $(\u,\v)\rhd (\s,\t)$.
\end{prop}

\begin{dfn} For any $\ell$-partition $\blam,$ we set $\alpha(\blam):=\sum_{s=1}^{\ell}\sum_{i\geq 1}(\lam_i^{(s)}-1)\lam_i^{(s)}/2$.
\end{dfn}

\begin{dfn}
For any $\blam\in\P_{\ell,n}$ and $\t\in\Std(\blam),$ we define
$$\begin{aligned}
&c_{\t,n}:=(-1)^{\ell-1}\frac{\gamma_{\ft}}{\gamma_{\ft_{\downarrow\leq(n-1)}}}
 \frac{\prod\limits_{\alpha\in \te{Rem}(\ft_{\downarrow\leq(n-1)})}\(\res_\ft(n)/\res(\alpha)-1\)}
{\prod\limits_{\gamma\in \te{Add}(\ft_{\downarrow\leq(n-1)})\backslash\{\ft^{-1}(n)\}}\(\res_\ft(n)/\res(\gamma)-1\)}\\
=\ &(-1)^{\ell-1}q^{n-n(\t)+c(\t,n)-1} \frac{\prod\limits_{\gamma\in \te{Add}(\ft_{\downarrow\leq(n-1)})\backslash\{\ft^{-1}(n)\}}\res(\gamma)}
{\prod\limits_{\alpha\in \te{Rem}(\ft_{\downarrow\leq(n-1)})}\res(\alpha)}
 \frac{\prod\limits_{\substack{\alpha\in \te{Rem}(\ft_{\downarrow\leq(n-1)})\\\alpha\not\prec\t^{-1}(n)}}\(\res_\ft(n)-\res(\alpha)\)}
{\prod\limits_{\substack{\gamma\in \te{Add}(\ft_{\downarrow\leq(n-1)})\backslash\{\ft^{-1}(n)\}\\
\gamma\not\prec\t^{-1}(n)}}\(\res_\ft(n)-\res(\gamma)\)}.
\end{aligned}$$
For each $\t\in \te{Std}(\P_{\ell,n}),$ we define $c_\t:=\prod_{k=1}^{n}c_{\t_{\downarrow\leq k},k}.$
\end{dfn}

\begin{prop}\label{nondegvefst}
Let $\blam\in\P_{\ell,n}$ and $\s,\t\in\te{Std}(\blam).$ Then:
\begin{itemize}
  \item[(1)] $\ve_{n}(\mff_{\fs\ft})=\delta_{n(\s),n(\t)}c_{\t,n}\mff_{\fs_{\downarrow\leq(n-1)}\ft_{\downarrow\leq(n-1)}};$
  \item[(2)] $\tau_R^{\te{MM}}(\mff_{\fs\ft})=\delta_{\s\t}c_\t.$
\end{itemize}
\end{prop}

\begin{proof}
For proving (1), first, we suppose that $n(\s)\neq n(\t)$. Then we deduce that
\begin{align}\label{shapenotequal}
\te{Shape}(\fs_{\downarrow\leq(n-1)})\neq \te{Shape}(\ft_{\downarrow\leq(n-1)}).
\end{align}
Now, for $1\leq k\leq n-1$, using Lemmas \ref{veAst} and \ref{fsemi}, we have
\begin{align}\label{eigenvaluesLk}
\mL_k\ve_n(\mff_{\fs\ft})=\ve_n(\mL_k\mff_{\fs\ft})=\res_\s(n)\ve_n(\mff_{\fs\ft}),\quad \ve_n(\mff_{\fs\ft})\mL_k=\ve_n(\mff_{\fs\ft}\mL_k)=\res_\t(n)\ve_n(\mff_{\fs\ft}).
\end{align}
Then combining (\ref{shapenotequal}), (\ref{eigenvaluesLk}) and Lemma \ref{fsemi}, we deduce that
$\ve_n(\mff_{\fs\ft})=0.$
On the other hand, if $n(\s)=n(\t)$, then the result follows from a similar argument as in the proof of \cite[Proposition 3.7]{HLL},
which gives the computation of $\ve_n(\mff_{\t\t}).$

The part (1) together with Lemma \ref{nondegtauoperation} imply the part (2). This completes the proof.
\end{proof}

The following corollaries give some equations  on the coefficients
which occur in the transition matrix between the Murphy basis elements and the seminormal basis elements of $\HH_{\ell,n}$.

\begin{cor}\label{maincor11}
Let $\blam\in\P_{\ell,n}$ and $\s,\t\in\te{Std}(\blam).$
\begin{enumerate}
  \item If $n(\s)= n(\t),$ then
for each $(\fp,\fq)\in \Std^2(\P_{\ell,n-1})$ with $(\fp,\fq)\rhd (\s_{\downarrow\leq (n-1)},\t_{\downarrow\leq (n-1)}),$ we have
$$\begin{aligned}
&q^{n-n(\t)}\Biggr(\prod\limits_{i=l(\t,n)+1}^{\ell}(-Q_i)\Biggr)+
\sum_{\substack{(\u,\v)\in\te{Std}^2(\P_{\ell,n})\\(\u,\v)\rhd (\s,\t)
\\\u_{\downarrow\leq (n-1)}=\s_{\downarrow\leq (n-1)}\\
\v_{\downarrow\leq (n-1)}=\t_{\downarrow\leq (n-1)}}}
a_{\u\v}^{\s\t}q^{n-n(\u)}\Biggr(\prod\limits_{i=l(\u,n)+1}^{\ell}(-Q_i)\Biggr)=c_{\t,n},\\
&\sum_{\substack{(\u,\v)\in\te{Std}^2(\P_{\ell,n})\\(\u,\v)\rhd (\s,\t)\\
\u_{\downarrow\leq (n-1)}=\fp\\
\v_{\downarrow\leq (n-1)}=\fq}} a_{\u\v}^{\s\t}q^{n-n(\u)}
\Biggr(\prod\limits_{i=l(\u,n)+1}^{\ell}(-Q_i)\Biggr)=c_{\t,n} a_{\fp\fq}^{\s_{\downarrow\leq (n-1)}\t_{\downarrow\leq (n-1)}}.
\end{aligned} $$
  \item If $n(\s)\neq n(\t),$ then for each $(\fp,\fq)\in \Std^2(\P_{\ell,n-1})$
  with $(\fp,\fq)\rhd (\s_{\downarrow\leq (n-1)},\t_{\downarrow\leq (n-1)}),$ we have
$$\sum_{\substack{(\u,\v)\in\te{Std}^2(\P_{\ell,n})\\(\u,\v)\rhd (\s,\t)\\
\u_{\downarrow\leq (n-1)}=\fp\\
\v_{\downarrow\leq (n-1)}=\fq}} a_{\u\v}^{\s\t}q^{n-n(\u)}
\Biggr(\prod\limits_{i=l(\u,n)+1}^{\ell}(-Q_i)\Biggr)=0. $$
\end{enumerate}
\end{cor}

\begin{proof}
First, using Propositions \ref{nondegbasestranslationrela} and \ref{MainProp1}, we have
\begin{align}\label{nondegcompu12}
\ve_n(\ff_{\s\t})
&=\ve_n\Big(\fm_{\s\t}+\sum_{\substack{(\u,\v)\in\te{Std}^2(\P_{\ell,n})\\(\u,\v)\rhd (\s,\t)}}a_{\u\v}^{\s\t}\fm_{\u\v}\Big)\\
\notag &=\delta_{n(\s),n(\t)} q^{n-n(\t)}\Biggr(\prod\limits_{i=l(\t,n)+1}^{\ell}(-Q_i)\Biggr)\fm_{\s_{\downarrow\leq (n-1)}\t_{\downarrow\leq (n-1)}}\\
\notag &\quad +\sum_{\substack{(\u,\v)\in\te{Std}^2(\P_{\ell,n})\\ n(\u)=n(\v)\\ (\u,\v)\rhd (\s,\t)}} a_{\u\v}^{\s\t}q^{n-n(\u)}
\Biggr(\prod\limits_{i=l(\u,n)+1}^{\ell}(-Q_i)\Biggr)\fm_{\u_{\downarrow\leq (n-1)}\v_{\downarrow\leq (n-1)}}.
\end{align}
On the other hand, by Propositions \ref{nondegvefst} and \ref{nondegbasestranslationrela}, we have
\begin{align}\label{nondegcompu13}
&\ve_n(\ff_{\s\t})=  c_{\t,n}\delta_{n(\s),n(\t)}\mff_{\fs_{\downarrow\leq(n-1)}\ft_{\downarrow\leq(n-1)}}\\
\notag =\ &\delta_{n(\s),n(\t)} c_{\t,n}\Big(\fm_{\s_{\downarrow\leq (n-1)}\t_{\downarrow\leq (n-1)}}
+\sum_{\substack{(\fp,\fq)\in\te{Std}^2(\P_{\ell,n-1})\\(\fp,\fq)\rhd (\s_{\downarrow\leq (n-1)},\t_{\downarrow\leq (n-1)})}}
a_{\fp\fq}^{\s_{\downarrow\leq (n-1)}\t_{\downarrow\leq (n-1)}}\fm_{\fp\fq}\Big).
\end{align}
Then the corollary follows from (\ref{nondegcompu12}) and (\ref{nondegcompu13}).
\end{proof}

\begin{cor}\label{maincor12}
Let $\blam\in\P_{\ell,n}$ and $\s,\t\in\te{Std}(\blam).$
\begin{enumerate}
  \item If $n(\s)= n(\t),$ then
for each $(\fp,\fq)\in \Std^2(\P_{\ell,n-1})$ with $(\fp,\fq)\rhd (\s_{\downarrow\leq (n-1)},\t_{\downarrow\leq (n-1)}),$ we have
$$\begin{aligned}
&c_{\t,n}+\sum_{\substack{(\u,\v)\in\te{Std}^2(\P_{\ell,n})\\(\u,\v)\rhd (\s,\t)
\\\u_{\downarrow\leq (n-1)}=\s_{\downarrow\leq (n-1)}\\
\v_{\downarrow\leq (n-1)}=\t_{\downarrow\leq (n-1)}}}
b_{\u\v}^{\s\t}c_{\v,n}
=q^{n-n(\t)}\Biggr(\prod\limits_{i=l(\t,n)+1}^{\ell}(-Q_i)\Biggr),\\
&\sum_{\substack{(\u,\v)\in\te{Std}^2(\P_{\ell,n})\\(\u,\v)\rhd (\s,\t)\\
\u_{\downarrow\leq (n-1)}=\fp\\
\v_{\downarrow\leq (n-1)}=\fq}} b_{\u\v}^{\s\t} c_{\v,n}=
q^{n-n(\t)}\Biggr(\prod\limits_{i=l(\t,n)+1}^{\ell}(-Q_i)\Biggr) b_{\fp\fq}^{\s_{\downarrow\leq (n-1)}\t_{\downarrow\leq (n-1)}}.
\end{aligned} $$
  \item If $n(\s)\neq n(\t),$ then for each $(\fp,\fq)\in \Std^2(\P_{\ell,n-1})$
  with $(\fp,\fq)\rhd (\s_{\downarrow\leq (n-1)},\t_{\downarrow\leq (n-1)}),$ we have
$$\sum_{\substack{(\u,\v)\in\te{Std}^2(\P_{\ell,n})\\(\u,\v)\rhd (\s,\t)\\
\u_{\downarrow\leq (n-1)}=\fp\\
\v_{\downarrow\leq (n-1)}=\fq}} b_{\u\v}^{\s\t} c_{\v,n}=0. $$
\end{enumerate}
\end{cor}

\begin{proof}
First, using Propositions \ref{nondegbasestranslationrela} and \ref{nondegvefst}, we have
\begin{align}\label{nondegcompu14}
\ve_n(\fm_{\s\t})
&=\ve_n\Big(\ff_{\s\t}+\sum_{\substack{(\u,\v)\in\te{Std}^2(\P_{\ell,n})\\(\u,\v)\rhd (\s,\t)}}b_{\u\v}^{\s\t}\ff_{\u\v}\Big)\\
\notag &=\delta_{n(\s),n(\t)} c_{\t,n}\ff_{\s_{\downarrow\leq (n-1)}\t_{\downarrow\leq (n-1)}}
+\sum_{\substack{(\u,\v)\in\te{Std}^2(\P_{\ell,n})\\ n(\u)=n(\v)\\ (\u,\v)\rhd (\s,\t)}}
b_{\u\v}^{\s\t}c_{\v,n}\ff_{\u_{\downarrow\leq (n-1)}\v_{\downarrow\leq (n-1)}}.
\end{align}
On the other hand, by Propositions \ref{MainProp1} and \ref{nondegbasestranslationrela}, we have
\begin{align}\label{nondegcompu15}
\ve_n(\fm_{\s\t}) &=\delta_{n(\s),n(\t)} q^{n-n(\t)}\Biggr(\prod\limits_{i=l(\t,n)+1}^{\ell}(-Q_i)\Biggr)\fm_{\s_{\downarrow\leq (n-1)}\t_{\downarrow\leq (n-1)}}\\
\notag &=\delta_{n(\s),n(\t)}q^{n-n(\t)}\Biggr(\prod\limits_{i=l(\t,n)+1}^{\ell}(-Q_i)\Biggr)\\
\notag &\quad\, \times \Big(\ff_{\s_{\downarrow\leq (n-1)}\t_{\downarrow\leq (n-1)}}
+\sum_{\substack{(\fp,\fq)\in\te{Std}^2(\P_{\ell,n-1})\\(\fp,\fq)\rhd (\s_{\downarrow\leq (n-1)},\t_{\downarrow\leq (n-1)})}}
b_{\fp\fq}^{\s_{\downarrow\leq (n-1)}\t_{\downarrow\leq (n-1)}}\ff_{\fp\fq}\Big).
\end{align}
Then the corollary follows from (\ref{nondegcompu14}) and (\ref{nondegcompu15}).
\end{proof}

For each $\t\in\Std(\P_{\ell,n})$, for convenience, we denote that
$$d_\t:=q^{\sum\limits_{k=1}^n(k-k(\t_{\downarrow\leq k}))}\Biggr(\prod\limits_{j=1}^n\prod\limits_{i=l(\t,j)+1}^{\ell}(-Q_i)\Biggr). $$
\begin{cor}\label{maincor13}
Let $\blam\in\P_{\ell,n}$ and $\s,\t\in\te{Std}(\blam).$ Then we have
\begin{align*}
\delta_{\s\t}d_\t+\sum_{\substack{\u\in\te{Std}(\P_{\ell,n})\\(\u,\u)\rhd (\s,\t)}}a_{\u\u}^{\s\t}d_\u=\delta_{\s\t} c_\t,\quad
\delta_{\s\t}c_\t+\sum_{\substack{\u\in\te{Std}(\P_{\ell,n})\\(\u,\u)\rhd (\s,\t)}}b_{\u\u}^{\s\t}c_\u=\delta_{\s\t} d_\t.
\end{align*}
\end{cor}

\begin{proof}
Using Propositions \ref{nondegbasestranslationrela} and \ref{nondegvefst} and Theorem \ref{mainthm1}, we have
\begin{align*}
&\delta_{\s\t} c_\t=\tau_R^{\te{MM}}(\ff_{\s\t})
=\tau_R^{\te{MM}}\Big(\fm_{\s\t}+\sum_{\substack{(\u,\v)\in\te{Std}^2(\P_{\ell,n})\\(\u,\v)\rhd (\s,\t)}}a_{\u\v}^{\s\t}\fm_{\u\v}\Big)
=\delta_{\s\t}d_\t+\sum_{\substack{\u\in\te{Std}(\P_{\ell,n})\\(\u,\u)\rhd (\s,\t)}}a_{\u\u}^{\s\t}d_\u,\\
&\delta_{\s\t} d_\t=\tau_R^{\te{MM}}(\fm_{\s\t})
=\tau_R^{\te{MM}}\Big(\ff_{\s\t}+\sum_{\substack{(\u,\v)\in\te{Std}^2(\P_{\ell,n})\\(\u,\v)\rhd (\s,\t)}}b_{\u\v}^{\s\t}\ff_{\u\v}\Big)
=\delta_{\s\t}c_\t+\sum_{\substack{\u\in\te{Std}(\P_{\ell,n})\\(\u,\u)\rhd (\s,\t)}}b_{\u\u}^{\s\t}c_\u.
\end{align*}
This completes the proof.
\end{proof}

\bigskip
\section{Proof of Theorem \ref{mainthm2}}

In this section, we will give the proof of the second main result in this paper---Theorem \ref{mainthm2}. That is, we shall compute the explicit value of the standard symmetrizing form $\tau_{R}^{\rm{BK}}$ of the degenerate cyclotomic Hecke algebra $H_{\ell,n}$ on each Murphy basis element $m_{\s\t}$ of $H_{\ell,n}$.
 Throughout this section, unless otherwise stated, $R$ is an integral domain, and $\bu=(u_1,u_2,\dots,u_\ell)\in R^{\ell}$.

\subsection{Computing the explicit value of $\tau_R^{\rm{BK}}$ on $m_{\s\t}$}
In this subsection, we shall use the cyclotomic Mackey decomposition
to compute the explicit value of $\tau_R^{\rm{BK}}$ on each Murphy basis element $m_{\s\t}$.

\begin{prop}\label{degvemst} Let $\blam\in\mathscr{P}_{\ell,n}$ and $\s,\t\in\te{Std}(\blam)$. Suppose that either $n(\s)=n$ or $n(\t)=n$. Then
$$\e_n(m_{\s\t})=
\begin{cases}
m_{\s_{\downarrow\leq (n-1)}\t_{\downarrow\leq (n-1)}}, & \mbox{if\ \,}n(\s)=n(\t)\te{\ \,and\ \,}l(\t,n)=1,  \\
0, & \mbox{otherwise}.
\end{cases}$$
\end{prop}

\begin{proof} We use $\ell_0$ to denote the largest integer such that $1\leq \ell_0\leq\ell$ and $\lam^{(\ell_0)}\neq\emptyset$. We prove the proposition by dividing it into three cases.

\smallskip\noindent
{\it Case 1.} Suppose $n(\s)=n=n(\t)$. Then we have $d(\s)=d(\s_{\downarrow\leq (n-1)})$, $d(\t)=d(\t_{\downarrow\leq (n-1)})$, $l(\t,n)=\ell_0$, and $\te{Shape}(\s_{\downarrow\leq (n-1)})=\te{Shape}(\t_{\downarrow\leq (n-1)})=\blam_{\downarrow\leq (n-1)}.$
First we let $w\in\Sym_\blam$ with $\sigma_{n-1}\leq w$.
Recalling (\ref{decompositionofwlamr}), we have $w_{\blam,{\ell_0},r_{\ell_0}}=\beta_{b_{{\ell_0},r_{\ell_0}},n-1}\cdot w_{\blam,{\ell_0},r_{\ell_0}}^{(1)}$.
Thus using (\ref{decompositionofdt})-(\ref{decompositionofwlamr}), (\ref{deg-L-comm1}) and Lemma \ref{deg-L-comm3} (1)
%and the equality $s_k(L_{k+1}-u)=(L_k-u)s_k+1$
 we deduce that
\begin{align*}
&d(\s)^*w u_\blam^+ d(\t)\\
=\ &d(\s_{\downarrow\leq (n-1)})^* \Biggr(\prod_{i=1}^{\ell_0-1} w_{i}\Biggr)\Biggr(\prod_{j=1}^{r_{\ell_0}-1}w_{\ell_0,j}\Biggr)
\beta_{b_{\ell_0,r_{\ell_0}},n-1}w_{\ell_0,r_{\ell_0}}^{(1)}
\Biggr(\prod_{s=2}^{\ell}\prod_{k=1}^{\fa_s}(L_k-u_s)\Biggr) d(\t_{\downarrow\leq (n-1)})\\
=\ &\varpi_{\s,\t,w}^{(1,1)} s_{n-1}\varpi^{(1,2)}_{\s,\t,w} + \rp^{(1)}_{\s,\t,w},
\end{align*}
where
$$\begin{aligned}
\varpi_{\s,\t,w}^{(1,1)}&=
d(\s_{\downarrow\leq (n-1)})^* \Biggr(\prod\limits_{i=1}^{\ell_0-1} w_{i}\Biggr)\Biggr(\prod\limits_{j=1}^{r_{\ell_0}-1}w_{\ell_0,j}\Biggr)
 \beta_{b_{\ell_0,r_{\ell_0}},n-2}\Biggr(\prod\limits_{s=\ell_0+1}^{\ell}(L_{n-1}-u_s)\Biggr)\in H_{\ell,n-1},\\
\varpi_{\s,\t,w}^{(1,2)}&=
w_{\ell_0,r_{\ell_0}}^{(1)}
\Biggr(\prod\limits_{s=2}^{\ell}\prod\limits_{\substack{1\leq k\leq \fa_s\\k\neq n}}(L_k-u_s)\Biggr) d(\t_{\downarrow\leq (n-1)})\in H_{\ell,n-1},\end{aligned}$$ and  $$\begin{aligned}
\rp^{(1)}_{\s,\t,w}&=
\begin{cases} \begin{matrix}
d(\s_{\downarrow\leq (n-1)})^* \bigg(\prod\limits_{i=1}^{\ell_0-1} w_{i}\bigg)\bigg(\prod\limits_{j=1}^{r_{\ell_0}-1}w_{\ell_0,j}\bigg)
\beta_{b_{\ell_0,r_{\ell_0}},n-2}\\
 \times \bigg(\sum\limits_{i=\ell_0+1}^{\ell}
\bigg( \prod\limits_{\ell_0+1\leq s<i}(L_{n-1}-u_s)
\prod\limits_{i< t\leq \ell}(L_{n}-u_t)\bigg)\bigg)\\
\times w_{\ell_0,r_{\ell_0}}^{(1)} \bigg(\prod\limits_{s=2}^{\ell}\prod\limits_{\substack{1\leq k\leq \fa_s\\k\neq n}}(L_k-u_s)\bigg)
 d(\t_{\downarrow\leq (n-1)}),\end{matrix} & \mbox{if\ \ } \ell_0<\ell,\\
  0, & \mbox{if\ \ } \ell_0=\ell.
\end{cases}
\end{aligned}$$
Then, noticing that
\begin{align}\label{easyobservationdeg}
\forall\,\ell_0+1\leq i\leq \ell,\ \te{the degree of\ }L_n\te{ in }\prod\limits_{i< t\leq {\ell}}(L_{n}-u_t) \te{\ lies in\ }[0,{\ell}-2],
\end{align}
we see that in this case
\begin{equation}\label{Case11Vedegn}
\e_n\bigl({d(\s)}^*w u_\blam^+ {d(\t)}\bigr)=\rp_{\ell-1}^{(n)}(\rp^{(1)}_{\s,\t,w})=0.
\end{equation}
On the other hand, for each $w\in\Sym_\blam$ with $\sigma_{n-1}\not\leq w$, using (\ref{deg-L-comm1}) and Lemma \ref{deg-L-comm3} (1) we have
\begin{align}\label{Case12Vedegn}
d(\s)^*w u_\blam^+ d(\t)&=d(\s_{\downarrow\leq (n-1)})^*w
\Biggr(\prod_{s=2}^{\ell}\prod_{\substack{1\leq k\leq \fa_s\\ k\neq n}}(L_k-u_s)\Biggr)d(\t_{\downarrow\leq (n-1)})
\prod\limits_{i=l(\t,n)+1}^{\ell}(L_{n}-u_i)\\
\notag &=d(\s_{\downarrow\leq (n-1)})^*w  u_{\blam_{\downarrow\leq (n-1)}}^+d(\t_{\downarrow\leq (n-1)})
\prod\limits_{i=l(\t,n)+1}^{\ell}(L_{n}-u_i).
\end{align}
Therefore, combining (\ref{Case11Vedegn}) and (\ref{Case12Vedegn}), we get that
\begin{align*}
\e_n(m_{\s\t})&=\rp_{\ell-1}^{(n)}(m_{\s\t})
=\sum_{\substack{w\in\Sym_\blam\\ \sigma_{n-1}\not\leq w}}\big(\delta_{l(\t,n),1}
d(\s_{\downarrow\leq (n-1)})^*w  u_{\blam_{\downarrow\leq (n-1)}}^+d(\t_{\downarrow\leq (n-1)}) \big)\\
\notag&=\delta_{l(\t,n),1} d(\s_{\downarrow\leq (n-1)})^*\bigg(\sum_{w\in\Sym_{\blam_{\downarrow\leq (n-1)}}}w\bigg)  u_{\blam_{\downarrow\leq (n-1)}}^+d(\t_{\downarrow\leq (n-1)}) \\
\notag &=\delta_{l(\t,n),1} d(\s_{\downarrow\leq (n-1)})^*
{\rm{x}}_{\blam_{\downarrow\leq (n-1)}}  u_{\blam_{\downarrow\leq (n-1)}}^+d(\t_{\downarrow\leq (n-1)}) \\
\notag&=\delta_{l(\t,n),1} m_{\s_{\downarrow\leq (n-1)}\t_{\downarrow\leq (n-1)}}.
\end{align*}

\smallskip\noindent
{\it Case 2.} Suppose $n(\s)\neq n=n(\t)$. Then $d(\s)\neq d(\s_{\downarrow\leq (n-1)})$ and $d(\t)=d(\t_{\downarrow\leq (n-1)}).$
Then for each $w\in\Sym_\blam$ such that $\sigma_{n-1}\not\leq w,$
using (\ref{decompositionofdt}), (\ref{deg-L-comm1}) and Lemma \ref{deg-L-comm3} (1) we deduce that
\begin{align}\label{degcompu3}
d(\s)^*w u_\blam^+ d(\t)
&=d(\s_{\downarrow\leq (n-1)})^*\gamma_{n-1,n(\s)} w
\Biggr(\prod_{s=2}^{\ell}\prod_{k=1}^{\fa_s}(L_k-u_s)\Biggr)d(\t_{\downarrow\leq (n-1)})\\
\notag &=\varpi_{\s,\t,w}^{(2,1)} s_{n-1}\varpi^{(2,2)}_{\s,\t,w} + \rp^{(2)}_{\s,\t,w},
\end{align}
where
$$\begin{aligned}
\varpi_{\s,\t,w}^{(2,1)}&=
d(\s_{\downarrow\leq (n-1)})^*
\Biggr(\prod\limits_{s=\ell_0+1}^{\ell}(L_{n-1}-u_s)\Biggr)\in H_{\ell,n-1},\\
\varpi_{\s,\t,w}^{(2,2)}&=\gamma_{n-2,n(\s)}w
\Biggr(\prod\limits_{s=2}^{\ell}\prod\limits_{\substack{1\leq k\leq \fa_s\\k\neq n}}(L_k-u_s)\Biggr) d(\t_{\downarrow\leq (n-1)})\in H_{\ell,n-1},\end{aligned}$$ and  $$\begin{aligned}
\rp^{(2)}_{\s,\t,w}&=
\begin{cases}\begin{matrix}
d(\s_{\downarrow\leq (n-1)})^* \bigg(\sum\limits_{i=\ell_0+1}^{\ell}
\bigg( \prod\limits_{\ell_0+1\leq s<i}(L_{n-1}-u_s)
\prod\limits_{i< t\leq \ell}(L_{n}-u_t)\bigg)\bigg) \\
\times \gamma_{n-2,n(\s)}w\bigg(\prod\limits_{s=2}^{\ell}\prod\limits_{\substack{1\leq k\leq \fa_s\\k\neq n}}(L_k-u_s)\bigg)d(\t_{\downarrow\leq (n-1)}),\end{matrix} & \mbox{if\ \ } \ell_0<\ell,\\
  0, & \mbox{if\ \ } \ell_0=\ell.
\end{cases}
\end{aligned}$$
As in Case 1, using (\ref{easyobservationdeg}) we see that in this case,
\begin{equation}\label{Case2Vendeg1}
\e_n\bigl({d(\s)}^*w u_\blam^+ {d(\t)}\bigr)=\rp_{\ell-1}^{(n)}(\rp^{(2)}_{\s,\t,w})=0.
\end{equation}
On the other hand, for each $w\in\Sym_\blam$ such that $\sigma_{n-1} \leq w,$
using (\ref{decompositionofdt})-(\ref{decompositionofwlamr}), (\ref{degbraidrela1}), (\ref{deg-L-comm1}), Lemma \ref{keytoolsdeg} (1) and Lemma \ref{deg-L-comm3} (1)
and noticing that $b_{\ell_0,r_{\ell_0}}\geq a_{\ell_0,r_{\ell_0}-1}+1>n(\s)$, we have
\begin{align*}
d(\s)^*w u_\blam^+ d(\t)
& =  d(\s_{\downarrow\leq (n-1)})^*\gamma_{n-1,n(\s)}
\Biggr(\prod_{i=1}^{\ell_0-1} w_{i}\Biggr)\Biggr(\prod_{j=1}^{r_{\ell_0}-1}w_{\ell_0,j}\Biggr)
\beta_{b_{\ell_0,r_{\ell_0},n-1}}w_{\ell_0,r_{\ell_0}}^{(1)} \\
\notag&\quad\ \times \Biggr(\prod_{s=2}^{\ell}\prod_{k=1}^{\fa_s}(L_k-Q_s)\Biggr)d(\t_{\downarrow\leq (n-1)})\\
&= d(\s_{\downarrow\leq (n-1)})^*\gamma_{n-1,n(\s)}\beta_{b_{\ell_0,r_{\ell_0},n-1}}
\Biggr(\prod_{i=1}^{\ell_0-1} w_{i}\Biggr)\Biggr(\prod_{j=1}^{r_{\ell_0}-1}w_{\ell_0,j}\Biggr)w_{\ell_0,r_{\ell_0}}^{(1)}\\
\notag&\quad\ \times
\Biggr(\prod_{s=2}^{\ell}\prod_{k=1}^{\fa_s}(L_k-u_s)\Biggr)d(\t_{\downarrow\leq (n-1)})\\
%\notag=\ &d(\s_{\downarrow\leq (n-1)})^*\(s_{b_{\ell_0,r_{\ell_0}}-1}s_{b_{\ell_0,r_{\ell_0}}}\cdots s_{n-2}\)\(s_{n-1}s_{n-2}\cdots s_{n(\s)}\)\\
%\notag&\times \Biggr(\prod_{i=1}^{\ell_0-1} w_{i}\Biggr)\Biggr(\prod_{j=1}^{r_{\ell_0}-1}w_{\ell_0,j}\Biggr) w_{\ell_0,r_{\ell_0}}^{(1)}
%\Biggr(\prod_{s=2}^{\ell}\prod_{k=1}^{\fa_s}(L_k-u_s)\Biggr)d(\t_{\downarrow\leq (n-1)})\\
&= d(\s_{\downarrow\leq (n-1)})^*\beta_{b_{\ell_0,r_{\ell_0}}-1,n-2}s_{n-1}\gamma_{n-2,n(\s)}
\Biggr(\prod_{i=1}^{\ell_0-1} w_{i}\Biggr)\Biggr(\prod_{j=1}^{r_{\ell_0}-1}w_{\ell_0,j}\Biggr) w_{\ell_0,r_{\ell_0}}^{(1)}\\
\notag&\quad\ \times
\Biggr(\prod_{s=2}^{\ell}\prod_{k=1}^{\fa_s}(L_k-u_s)\Biggr)d(\t_{\downarrow\leq (n-1)})\\
&= \varpi_{\s,\t,w}^{(3,1)} s_{n-1}\varpi^{(3,2)}_{\s,\t,w} + \rp^{(3)}_{\s,\t,w},
\end{align*}
where
$$\begin{aligned}
\varpi_{\s,\t,w}^{(3,1)}&=
d(\s_{\downarrow\leq (n-1)})^*\beta_{b_{\ell_0,r_{\ell_0}}-1,n-2}
\Biggr(\prod\limits_{s=l(\blam,n)+1}^{\ell}(L_{n-1}-u_s)\Biggr)\in H_{\ell,n-1},\\
\varpi_{\s,\t,w}^{(3,2)}&=\gamma_{n-2,n(\s)}
\Biggr(\prod_{i=1}^{\ell_0-1} w_{i}\Biggr)\Biggr(\prod_{j=1}^{r_{\ell_0}-1}w_{\ell_0,j}\Biggr) w_{\ell_0,r_{\ell_0}}^{(1)}
\Biggr(\prod\limits_{s=2}^{\ell}\prod\limits_{\substack{1\leq k\leq \fa_s\\k\neq n}}(L_k-u_s)\Biggr) d(\t_{\downarrow\leq (n-1)})\in H_{\ell,n-1},\end{aligned}$$ and $$\begin{aligned}
\rp^{(3)}_{\s,\t,w} =
\begin{cases}\begin{matrix}
d(\s_{\downarrow\leq (n-1)})^* \beta_{b_{\ell_0,r_{\ell_0}}-1,n-2}
 \bigg(\sum\limits_{i=\ell_0+1}^{\ell}\bigg( \prod\limits_{\ell_0+1\leq s<i}(L_{n-1}-u_s)\\
\times \prod\limits_{i< t\leq \ell}(L_{n}-u_t)\bigg)\bigg)
\gamma_{n-2,n(\s)}\bigg(\prod\limits_{i=1}^{\ell_0-1} w_{i}\bigg)\bigg(\prod\limits_{j=1}^{r_{\ell_0}-1}w_{\ell_0,j}\bigg) w_{\ell_0,r_{\ell_0}}^{(1)}\\
\times \bigg(\prod\limits_{s=2}^{\ell}\prod\limits_{\substack{1\leq k\leq \fa_s\\k\neq n}}(L_k-u_s)\bigg)d(\t_{\downarrow\leq (n-1)}),\end{matrix} & \mbox{if\ \ }\ell_0<\ell,\\
  0, & \mbox{if\ \ } \ell_0=\ell.
\end{cases}
\end{aligned}$$
Then still using (\ref{easyobservationdeg}) we see that in this case, \begin{equation}\label{Case2Vendeg2}
\e_n\bigr(d(\s)^*w u_\blam^+ d(\t)\bigr)=\rp_{\ell-1}^{(n)}(\rp^{(3)}_{\s,\t,w})=0.
\end{equation}
Therefore, combining (\ref{Case2Vendeg1}) and (\ref{Case2Vendeg2}), we get
$\e_n(m_{\s\t}) =0.$

\smallskip\noindent
{\it Case 3.} Suppose that $n(\s)=n\neq n(\t)$. Then it follows from Lemma \ref{veAst2}, (\ref{degmst*}) and the result of Case 2 that
$$\e_n(m_{\s\t})=\e_n(m_{\s\t}^*)=\e_n(m_{\t\s})=0.$$
 This completes the proof of the proposition.
\end{proof}

\begin{prop}\label{degvemst2} Let $\blam\in\mathscr{P}_{\ell,n}$ and $\s,\t\in\te{Std}(\blam)$. Suppose that $n(\s)\neq n\neq n(\t)$. Then
$$\e_n(m_{\s\t})=
\begin{cases}
m_{\s_{\downarrow\leq (n-1)}\t_{\downarrow\leq (n-1)}}, & \mbox{if\ \,}n(\s)=n(\t)\te{\ \,and\ \,}l(\t,n)=1,  \\
0, & \mbox{otherwise}.
\end{cases}$$\end{prop}

\begin{proof} We use $\ell_0$ to denote the largest integer such that $1\leq \ell_0\leq\ell$ and $\lam^{(\ell_0)}\neq\emptyset$.
By assumption, $\sigma_{n-1} \leq d(\s)$ and $\sigma_{n-1} \leq d(\t)$. In view of Lemma \ref{veAst2} and (\ref{degmst*}), we may assume that $n(\s)\leq n(\t)$.

For each $w\in\Sym_\blam$, using
(\ref{decompositionofdt})-(\ref{decompositionofwlamr}) and Lemma \ref{keytoolsdeg} (1) we deduce that
\begin{align}\label{degcompu5}
&d(\s)^*wu_\blam^+ d(\t)\\
\notag =\ &d(\s_{\downarrow\leq (n-1)})^*\gamma_{n-1,n(\s)} \Biggr(\prod_{j=1}^{\ell_0}w_{j}\Biggr)
\Biggr(\prod_{s=2}^{\ell}\prod_{i=1}^{\fa_s}(L_i-u_s)\Biggr)
\beta_{n(\t),n-1} d(\t_{\downarrow\leq (n-1)})\\
\notag =\ &d(\s_{\downarrow\leq (n-1)})^*\Biggr(\prod_{a=1}^{l(\s,n)-1} w_{a}\Biggr)
\Biggr(\prod_{b=1}^{r(\s,n)-1} w_{l(\s,n),b}\Biggr)
\Biggr(\prod_{c=r(\s,n)+1}^{r_{l(\s,n)}} (w_{l(\s,n),c})_{\downarrow} \Biggr)\\
\notag &\times \Biggr(\prod_{d=l(\s,n)+1}^{\ell_0} (w_{d})_{\downarrow} \Biggr)
 \gamma_{n-1,n(\s)}w_{l(\s,n),r(\s,n)} \Biggr(\prod_{s=2}^{\ell}\prod_{i=1}^{\fa_s}(L_i-u_s)\Biggr)\beta_{n(\t),n-1}d(\t_{\downarrow\leq (n-1)}).
\end{align}
We want to compute $\e_n({d(\s)}^*wu_\blam^+ {d(\t)})$. Since both the elements $$
d(\s_{\downarrow\leq (n-1)})^*\Biggr(\prod_{a=1}^{l(\s,n)-1} {w_{a}} \Biggr)
\Biggr(\prod_{b=1}^{r(\s,n)-1} {w_{l(\s,n),b}}\Biggr)
\Biggr(\prod_{c=r(\s,n)+1}^{r_{l(\s,n)}}{(w_{l(\s,n),c})_{\downarrow}} \Biggr)
 \Biggr(\prod_{d=l(\s,n)+1}^{\ell_0} {(w_{d})_{\downarrow}} \Biggr)$$
and ${d(\t_{\downarrow\leq (n-1)})}$ lie in $H_{\ell,n-1}$, in view of Lemma \ref{veAst2}, it suffices to consider the cyclotomic Mackey decomposition of \begin{equation*}\label{mackey111deg}
\gamma_{n-1,n(\s)}{w_{l(\s,n),r(\s,n)}}  \Biggr(\prod_{s=2}^{\ell}\prod_{i=1}^{\fa_s}(L_i-u_s)\Biggr){\beta_{n(\t),n-1}}.
\end{equation*}
and its value under $\e_n.$

Denote that $\bnu:=\te{Shape}(\t_{\downarrow\leq (n-1)})$.
By definition, $$
u_{\bnu}^+=\Biggr(\prod_{s=2}^{l(\t,n)}\prod_{i=1}^{\mathfrak{a}_s}(L_i-u_s)\Biggr)
\Biggr(\prod_{k=l(\t,n)+1}^{\ell}\prod_{j=1}^{\mathfrak{a}_s-1}(L_j-u_k)\Biggr)=\prod_{s=2}^{\ell}\prod_{\substack{1\leq i\leq\fa_s\\ i\neq n(\t)}}\(L_{\beta_{n(\t),n-1}^{-1}(i)}-u_s\).
 $$
Thus applying Lemma \ref{JucyMurphyCommutatorsdeg}, we get that
\begin{align*}
&\Biggr(\prod_{s=2}^{\ell}\prod_{i=1}^{\fa_s}(L_i-u_s)\Biggr)\beta_{n(\t),n-1}\\
=\ &\beta_{n(\t),n-1}\prod_{s=2}^{\ell}\prod_{i=1}^{\fa_s}\(L_{\beta_{n(\t),n-1}^{-1}(i)}-u_s\)
+\sum_{\substack{1\leq m\leq n-k\\k\leq i_1<i_2<\dots<i_m\leq n-1}} \beta^{i_1,\dots,i_m}_{k,n-1}f_{i_1,\dots,i_m}\\
=\ &\beta_{n(\t),n-1}\prod_{s=2}^{\ell}\prod_{\substack{1\leq i\leq\fa_s\\ i\neq n(\t)}}\(L_{\beta_{n(\t),n-1}^{-1}(i)}-u_s\)
\Biggr(\prod_{j=l(\t,n)+1}^{\ell}(L_{n}-u_j)\Biggr)\\
&+\sum_{\substack{1\leq m\leq n-k\\k\leq i_1<i_2<\dots<i_m\leq n-1}} \beta^{i_1,\dots,i_m}_{k,n-1}f_{i_1,\dots,i_m}\\
=\ &\beta_{n(\t),n-1}u_\bnu^+\Biggr(\prod_{i=l(\t,n)+1}^{\ell}(L_{n}-u_i)\Biggr)
+\sum_{\substack{1\leq m\leq n-k\\k\leq i_1<i_2<\dots<i_m\leq n-1}} \beta^{i_1,\dots,i_m}_{k,n-1}f_{i_1,\dots,i_m},
\end{align*}
where $f_{i_1,\dots,i_m}=\sum_{a=0}^{\ell-1}f_{i_1,\dots,i_m}^{(a)}L^a_n$
for some $f_{i_1,\dots,i_m}^{(a)}\in R[L_1,\dots,L_{n-1}]$.

Suppose that $\sigma_{n(\s)-1} \leq w_{l(\s,n),r(\s,n)}$. Note that Lemma \ref{deg-L-comm3} implies that
\begin{align}\label{propertysLcomm}
s_{n-1}L_n^k\in H_{\ell,n-1}s_{n-1}H_{\ell,n-1}+\sum_{j=0}^{\ell-2}H_{\ell,n-1}L_n^{j},\ \ \forall\,0\leq k\leq \ell-1.
\end{align}
Noticing that $c_{l(\s,n),r(\s,n)}\leq n(\s)-1< n(\t),$
and using (\ref{degbraidrela1}), (\ref{deg-L-comm1}), Lemma \ref{keytoolsdeg} (1) and (\ref{propertysLcomm}),
we deduce that: for each $1\leq m\leq n-n(\t)$ and $n(\t)\leq i_1<\dots<i_m\leq n-1$,
\begin{align}\label{degcompu7}
&\quad\,{\gamma_{n-1,n(\s)}}{w_{l(\s,n),r(\s,n)}}
{\beta^{i_1,\dots,i_m}_{n(\t),n-1}}f_{i_1,\dots,i_m}\\
\notag & ={\gamma_{n-1,n(\s)}}{w_{l(\s,n),r(\s,n)}^{(2)}}{\gamma_{n(\s)-1,c_{l(\s,n),r(\s,n)}}}
{\beta^{i_1,\dots,i_m}_{n(\t),n-1}}f_{i_1,\dots,i_m}\\
\notag & ={w_{l(\s,n),r(\s,n)}^{(2)}}
{\gamma_{n-1,c_{l(\s,n),r(\s,n)}}}{\beta^{i_1,\dots,i_m}_{n(\t),n-1}}
f_{i_1,\dots,i_m} \\
\notag & ={w_{l(\s,n),r(\s,n)}^{(2)}}
{(\beta^{i_1,\dots,i_m}_{n(\t),n-1})_{\downarrow}}{\gamma_{n-1,c_{l(\s,n),r(\s,n)}}}
f_{i_1,\dots,i_m}\\
\notag &= \mu_{s_{n-1}}(\varpi^{(4,i_1,\dots,i_m)}_{\s,\t,w})+\sum_{j=0}^{\ell-2}g_j^{i_1,\dots,i_m} L_n^j
\end{align}
for some $\varpi^{(4,i_1,\dots,i_m)}_{\s,\t,w}\in H_{\ell,n-1}\ot_{H_{\ell,n-2}}H_{\ell,n-1}$
and $g_j^{i_1,\dots,i_m}\in H_{\ell,n-1},\ 0\leq j\leq \ell-2$. Similarly, \begin{align}\label{degcompu8}
&\quad\, {\gamma_{n-1,n(\s)}}{w_{l(\s,n),r(\s,n)}}{\beta_{n(\t),n-1}}u_\bnu^+ \Biggr(\prod_{i=l(\t,n)+1}^{\ell}(L_{n}-u_i)\Biggr) \\
\notag &={\gamma_{n-1,n(\s)}}{w_{l(\s,n),r(\s,n)}^{(2)}}{\gamma_{n(\s)-1,c_{l(\s,n),r(\s,n)}}}{\beta_{n(\t),n-1}}u_\bnu^+ \Biggr(\prod_{i=l(\t,n)+1}^{\ell}(L_{n}-u_i)\Biggr)\\
 \notag &= {w_{l(\s,n),r(\s,n)}^{(2)}}
 {\gamma_{n-1,c_{l(\s,n),r(\s,n)}}}{\beta_{n(\t),n-1}}u_\bnu^+ \Biggr(\prod_{i=l(\t,n)+1}^{\ell}(L_{n}-u_i)\Biggr)\\
\notag &= {w_{l(\s,n),r(\s,n)}^{(2)}}{\beta_{n(\t)-1,n-2}}
 {\gamma_{n-1,c_{l(\s,n),r(\s,n)}}}u_\bnu^+ \Biggr(\prod_{i=l(\t,n)+1}^{\ell}(L_{n}-u_i)\Biggr)\\
\notag &={w_{l(\s,n),r(\s,n)}^{(2)}}{\beta_{n(\t)-1,n-2}}
s_{n-1}\Biggr(\prod_{i=l(\t,n)+1}^{\ell}(L_{n}-u_i)\Biggr)\gamma_{n-2,c_{l(\s,n),r(\s,n)}}u_\bnu^+\\
\notag &= \mu_{s_{n-1}}(\varpi^{(5)}_{\s,\t,w})+\sum_{j=0}^{\ell-2}g_j  L_n^j
\end{align}
for some $\varpi^{(5)}_{\s,\t,w}\in H_{\ell,n-1}\ot_{H_{\ell,n-2}}H_{\ell,n-1}$
and $g_j \in H_{\ell,n-1},\ 0\leq j\leq \ell-2$.
For (\ref{degcompu7}) and (\ref{degcompu8}), their values under $\e_n$ are both zero by the cyclotomic Mackey decomposition (\ref{degMadecomequation}). Therefore, if $\sigma_{n(\s)-1} \leq w_{l(\s,n),r(\s,n)}$, then $\e_n({d(\s)}^*wu_\blam^+ {d(\t)})=0$.

Now assume that $\sigma_{n(\s)-1}\not\leq w_{l(\s,n),r(\s,n)}$, which implies that $w_{l(\s,n),r(\s,n)}$ commutes with $\gamma_{n-1,n(\s)}$.
Then using the assumption $n(\s)\leq n(\t)$,
Lemma \ref{keytoolsdeg}, (\ref{deg-L-comm1}) and (\ref{propertysLcomm}),
we deduce that: for each $1\leq m\leq n-n(\t)$ and $n(\t)\leq i_1<\dots<i_m\leq n-1$,
\begin{align}\label{degcompu9}
&{\gamma_{n-1,n(\s)}}{w_{l(\s,n),r(\s,n)}}
{\beta^{i_1,\dots,i_m}_{n(\t),n-1}}f_{i_1,\dots,i_m}{d(\t_{\downarrow\leq (n-1)})}\\
\notag =\ &w_{l(\s,n),r(\s,n)}
 {\gamma_{n-1,n(\s)}}{\beta^{i_1,\dots,i_m}_{n(\t),n-1}}f_{i_1,\dots,i_m}{d(\t_{\downarrow\leq (n-1)})}\\
\notag =\ &\mu_{s_{n-1}}(\varpi^{(6,i_1,\dots,i_m)}_{\s,\t,w})+\sum_{j=0}^{\ell-2}h_j^{i_1,\dots,i_m} L_n^j
\end{align}
for some $\varpi^{(6,i_1,\dots,i_m)}_{\s,\t,w}\in H_{\ell,n-1}\ot_{H_{\ell,n-2}}H_{\ell,n-1}$
and $h_j^{i_1,\dots,i_m}\in H_{\ell,n-1},\ 0\leq j\leq \ell-2$. So, once again, the value of this term under $\e_n$ is zero by the cyclotomic Mackey decomposition (\ref{degMadecomequation}).

Therefore, it remains to consider the following term under the assumption that $\sigma_{n(\s)-1}\not\leq w_{l(\s,n),r(\s,n)}$. Noticing that in this case $w_{l(\s,n),r(\s,n)}$ commutes with ${\gamma_{n-1,n(\s)}}$, we get that
 \begin{align}\label{wlsnrsncommdeg}
  & {\gamma_{n-1,n(\s)}}{w_{l(\s,n),r(\s,n)}}{\beta_{n(\t),n-1}}
   u_\bnu^+ \Biggr(\prod_{i=l(\t,n)+1}^{\ell}(L_{n}-u_i)\Biggr){d(\t_{\downarrow\leq (n-1)})} \\
\notag=\ &{w_{l(\s,n),r(\s,n)}}
   {\gamma_{n-1,n(\s)}}{\beta_{n(\t),n-1}}\Biggr(\prod_{i=l(\t,n)+1}^{\ell}(L_{n}-u_i)\Biggr)u_\bnu^+ {d(\t_{\downarrow\leq (n-1)})} .
\end{align}

If $n(\s)< n(\t)$, then using Lemma \ref{keytoolsdeg} (1), we have $$\begin{aligned}
&{w_{l(\s,n),r(\s,n)}}{\gamma_{n-1,n(\s)}}{\beta_{n(\t),n-1}}\Biggr(\prod_{i=l(\t,n)+1}^{\ell}(L_{n}-u_i)\Biggr)\\
=\ &{w_{l(\s,n),r(\s,n)}}{\beta_{n(\t)-1,n-2}}s_{n-1}\Biggr(\prod_{i=l(\t,n)+1}^{\ell}(L_{n}-u_i)\Biggr){\gamma_{n-2,n(\s)}}\\
=\ &\mu_{s_{n-1}}(\varpi_{\s\t})+\sum_{j=0}^{\ell-2}h_jL_n^j
\end{aligned}
$$
for some $\varpi_{\s\t}\in H_{\ell,n-1}\ot_{H_{\ell,n-2}}H_{\ell,n-1}$ and $h_j\in H_{\ell,n-1},\ 0\leq j\leq \ell-2$. So, in this case,
the value of (\ref{wlsnrsncommdeg}) under $\e_n$ is zero by the cyclotomic Mackey decomposition (\ref{degMadecomequation}).

If $n(\s)=n(\t)$, the relation (\ref{degquadraticrela}) implies that
${\gamma_{n-1,n(\s)}}{\beta_{n(\t),n-1}}=1$. Then it follows from this, (\ref{degcompu5}) and (\ref{wlsnrsncommdeg}) that in this case $$\begin{aligned}
\e_n({d(\s)}^*wu_\blam^+ {d(\t)})
&=\delta_{l(\t,n),1} {d(\s_{\downarrow\leq (n-1)})}^*\Biggr(\prod_{a=1}^{l(\s,n)-1} {w_{a}} \Biggr)
\Biggr(\prod_{b=1}^{r(\s,n)} {w_{l(\s,n),b}}\Biggr)\\
&\qquad   \times \Biggr(\prod_{c=r(\s,n)+1}^{r_{l(\s,n)}} {(w_{l(\s,n),c})_{\downarrow}} \Biggr)
\Biggr(\prod_{d=l(\s,n)+1}^{\ell_0}{(w_{d})_{\downarrow}} \Biggr)u_\bnu^+ {d(\t_{\downarrow\leq (n-1)})} .
\end{aligned}$$

Note that if $n(\s)=n(\t)$, then $\te{Shape}(\s_{\downarrow\leq (n-1)})=\te{Shape}(\t_{\downarrow\leq (n-1)})=\bnu. $
Hence in this case we have (\ref{Sbnu}) and
\begin{align}\label{degcompu11}
&\sum_{\substack{w\in\Sym_\blam\\ \sigma_{n(\s)-1}\not\leq w_{l(\s,n),r(\s,n)}}}\Biggr(\prod\limits_{a=1}^{l(\s,n)-1} {w_{a}}\Biggr)
\Biggr(\prod\limits_{b=1}^{r(\s,n)} {w_{l(\s,n),b}}\Biggr)
\Biggr(\prod\limits_{c=r(\s,n)+1}^{r_{l(\s,n)}} {(w_{l(\s,n),c})_{\downarrow}} \Biggr)\\
\notag &\qquad \qquad \qquad \times \Biggr(\prod\limits_{d=l(\s,n)+1}^{\ell_0} {(w_{d})_{\downarrow}} \Biggr)
 = \sum_{u\in \Sym_\bnu}u={\rm{x}}_\bnu.
\end{align}
Therefore, summarizing all the discussion above, we get that: if $n(\s)\neq n(\t),$ then $\e_n(m_{\s\t})=0;$ if $n(\s)=n(\t)$, then
\begin{align*}
 \e_n(m_{\s\t})
&=\delta_{l(\t,n),1}\sum_{\substack{w\in\Sym_\blam\\ \sigma_{n(\s)-1}\not\leq w_{l(\s,n),r(\s,n)}}}
{d(\s_{\downarrow\leq (n-1)})}^*
\Biggr(\prod\limits_{a=1}^{l(\s,n)-1} {w_{a}} \Biggr)\Biggr(\prod\limits_{b=1}^{r(\s,n)} {w_{l(\s,n),b}}\Biggr)\\
&\qquad \qquad \qquad \times \Biggr(\prod\limits_{c=r(\s,n)+1}^{r_{l(\s,n)}} {(w_{l(\s,n),c})_{\downarrow}} \Biggr)
  \Biggr(\prod\limits_{d=l(\s,n)+1}^{\ell_0} {(w_{d})_{\downarrow}} \Biggr) u_\bnu^+ {d(\t_{\downarrow\leq (n-1)})}\\
&=\delta_{l(\t,n),1}{d(\s_{\downarrow\leq (n-1)})}^*
{\rm{x}}_\bnu u^+_\bnu {d(\t_{\downarrow\leq (n-1)})}\\
& =\delta_{l(\t,n),1} m_{\s_{\downarrow\leq (n-1)}\t_{\downarrow\leq (n-1)}}.
\end{align*}
This completes the proof of the proposition.
\end{proof}

\begin{prop}\label{degvemst3}
Let $\blam\in\mathscr{P}_{\ell,n}$ and $\s,\t\in\te{Std}(\blam).$ Then
$$\e_n(m_{\s\t})=\begin{cases}  m_{\s_{\downarrow\leq (n-1)}\t_{\downarrow\leq (n-1)}}, &\text{if\ \,}n(\s)=n(\t)\te{\ \,and\ \,}l(\t,n)=1,\\
0, &\text{otherwise.}
\end{cases}$$
\end{prop}

\begin{proof}
This follows from Propositions \ref{degvemst} and \ref{degvemst2}.
\end{proof}

%%%For any $\s,\t\in\Std(\blam)$, it is clear that $\s\neq \t$ if and only if there is an integer $1\leq j\leq n$ such that $n(\s_{\downarrow\leq j})\neq n(\t_{\downarrow\leq j})$.
As a consequence, we get the proof of Theorem \ref{mainthm2}.\medskip

\noindent
{\bf Proof of Theorem \ref{mainthm2}:} This follows from Lemma \ref{degtoperation} and Proposition \ref{degvemst3}.
\hfill\qed

\subsection{An application}

In the rest of this section, we shall give an application of Proposition \ref{degvemst3} and Theorem \ref{mainthm2}. From now on until the end of this section, we shall assume that $R=K$ is a field and $H_{\ell,n}$ is semisimple over $K$.

Let $\blam\in\P_{\ell,n}$ and $\t\in\Std(\blam)$. For each $1\leq k\leq n$, we define the content of $k$ in $\t$ to be
$$c_{\t}(k)=(c-b)\cdot 1_K+u_{l}\te{\quad if\quad} \t^{-1}(k)=(b,c,l).$$
%if $k$ appears in row $b$ and column $c$ of $\t^{(l)}$, where $1\leq l\leq \ell$.
For each $1\leq k\leq n$, we define $\mathrm{C}(k):=\left\{c_\t(k)\mid \t \in \Std(\blam),\ \blam\in \P_{\ell, n}\right\}$.

\begin{dfn}\text{\rm (\cite{Mur92}, \cite[Definition 6.7]{AMR})}\label{dfn:Ft2}
For any $\s,\t\in\Std(\blam)$ and $\blam\in\P_{\ell,n}$, we define
$$\rF_{\t}:=\prod\limits^n\limits_{k=1}\prod\limits_{\substack{c\in \mathrm{C}(k)\\c\neq c_{\t}(k)}}\frac{L_k-c}{c_{\t}(k)-c}\in H_{\ell,n},\quad
f_{\s\t}:=\rF_\s m_{\s\t}\rF_\t\in H_{\ell,n}.
$$
\end{dfn}
By \cite{AMR}, the set $\{f_{\s\t}\mid \s,\t\in\Std(\blam),\ \blam\in\P_{\ell,n}\}$ gives a $K$-basis of $H_{\ell,n}$, and is called the seminormal basis of $H_{\ell,n}$. For $\t\in\te{Std}(\P_{\ell,n}),$ following \cite[Theorem 3.7]{Zhao}, we define
$$r_\t:=\prod_{i=1}^{n}\frac{\prod\limits_{\gamma\in \mathscr{A}_{\ft}(i)}\(c_\ft(i)-c(\gamma)\)}
{\prod\limits_{\alpha\in \mathscr{R}_{\ft}(i)}\(c_\ft(i)-c(\alpha)\)}.$$

\begin{lem}\text{\rm (\cite[Lemma 3.8, Proposition 3.14]{zdk})}\label{fsemi2}
\begin{itemize}
  \item[(1)] For any $\s,\t\in\Std(\blam)$ and $\u,\v\in\Std(\bmu),$ where $\blam,\bmu\in\P_{\ell,n},$ we have $$
{f}_{\s\t}{f}_{\u\v}=\delta_{\t\u}r_\t {f}_{\s\v}.
$$
  \item[(2)] Let $\blam\in\P_{\ell,n}$ and $\s,\t\in\Std(\blam)$. Then for each $1\leq k\leq n,$ $$
L_k {f}_{\s\t}=c_\s(k){f}_{\s\t},\quad  {f}_{\s\t}L_k=c_\t(k){f}_{\s\t}. $$
Moreover, $$
Kf_{\s\t}=\bigl\{h\in H_{\ell,n}\bigm|L_kh=c_\s(k)h,\ hL_k=c_\t(k)h,\ \forall\,1\leq k\leq n\bigr\}.
$$
  \item[(3)] For any $\s,\t\in\Std(\P_{\ell,n}),$ we have $\s=\t$ if and only if $c_\s(k)=c_\t(k),\ \forall\,1\leq k\leq n$.
\end{itemize}
%In particular, $F_\s:=\mathfrak{f}_{\s\s}/\gamma_\s$ is a primitive idempotent, and $$
%\tau_n^{\K}(\mathfrak{f}_{\s\t})=\delta_{\s\t}\frac{\gamma_{\s}}{s_\blam(q;\bQ)},$$
%where $s_\blam(q,\bQ)$ is the Schur element of $\HH_{\ell,n}^{\K}$ associated to $\blam$.
\end{lem}

\begin{prop}\te{(}\cite[Proposition 3.14]{zdk}\te{)}\label{degbasestranslationrela}
Let $\blam\in\P_{\ell,n}$ and $\s,\t\in\te{Std}(\blam).$ Then we have
$$f_{\s\t}=m_{\s\t}+\sum_{\substack{(\u,\v)\in\te{Std}^2(\P_{\ell,n})\\(\u,\v)\rhd (\s,\t)}}c_{\u\v}^{\s\t}m_{\u\v},\quad
m_{\s\t}=f_{\s\t}+\sum_{\substack{(\u,\v)\in\te{Std}^2(\P_{\ell,n})\\(\u,\v)\rhd (\s,\t)}}d_{\u\v}^{\s\t}f_{\u\v}, $$
where $c_{\u\v}^{\s\t},d_{\u\v}^{\s\t}\in K$ for each $(\u,\v)\in\te{Std}^2(\P_{\ell,n})$ with $(\u,\v)\rhd (\s,\t)$.
\end{prop}

\begin{dfn}
For any $\blam\in\P_{\ell,n}$ and $\t\in\Std(\blam)$, we define
$$\begin{aligned}
{\rm{c}}_{\t,n}&:=\frac{r_{\ft}}{r_{\ft_{\downarrow\leq(n-1)}}}
 \frac{\prod\limits_{\alpha\in \te{Rem}(\ft_{\downarrow\leq(n-1)})}\(c_\ft(n)-c(\alpha)\)}
{\prod\limits_{\gamma\in \te{Add}(\ft_{\downarrow\leq(n-1)})\backslash\{\ft^{-1}(n)\}}\(c_\ft(n)-c(\gamma)\)}=\frac{\prod\limits_{\substack{\alpha\in \te{Rem}(\ft_{\downarrow\leq(n-1)})\\\alpha\not\prec\t^{-1}(n)}}\(c_\ft(n)-c(\alpha)\)}
{\prod\limits_{\substack{\gamma\in \te{Add}(\ft_{\downarrow\leq(n-1)})\backslash\{\ft^{-1}(n)\}\\
\gamma\not\prec\t^{-1}(n)}}\(c_\ft(n)-c(\gamma)\)}.
\end{aligned}$$
For each $\t\in \te{Std}(\P_{\ell,n}),$ we define ${\rm{c}}_\t:=\prod_{k=1}^{n}{\rm{c}}_{\t_{\downarrow\leq k},k}.$
\end{dfn}

\begin{prop}\label{degvefst3}
Let $\blam\in\P_{\ell,n}$ and $\s,\t\in\te{Std}(\blam).$ Then:
\begin{itemize}
  \item[(1)] $\e_{n}(f_{\fs\ft})=\delta_{n(\s),n(\t)}{\rm{c}}_{\t,n}f_{\fs_{\downarrow\leq(n-1)}\ft_{\downarrow\leq(n-1)}};$
  \item[(2)]  $\tau_R^{\te{BK}}(f_{\fs\ft})=\delta_{\s\t}{\rm{c}}_\t.$
\end{itemize}
\end{prop}

\begin{proof}
For (1), if $n(\s)\neq n(\t)$, then the result follows from Lemmas \ref{veAst2} and \ref{fsemi2}, and a similar argument as in the proof of Lemma \ref{nondegvefst}.
On the other hand, if $n(\s)= n(\t)$, then the part (1) follows from a similar argument as in the proof of \cite[Proposition 4.7]{HLL},
which gives the computation of $\e_n(f_{\t\t})$.

The part (1) together with Lemma \ref{degtoperation} imply the part (2). This completes the proof of this proposition.
\end{proof}

The following corollaries give some equations  on the coefficients
which occur in the transition matrix between the Murphy basis elements and the seminormal basis elements of $H_{\ell,n}$.

\begin{cor}
Let $\blam\in\P_{\ell,n}$ and $\s,\t\in\te{Std}(\blam).$
\begin{enumerate}
  \item If $n(\s)= n(\t),$ then
for each $(\fp,\fq)\in \Std^2(\P_{\ell,n-1})$ with $(\fp,\fq)\rhd (\s_{\downarrow\leq (n-1)},\t_{\downarrow\leq (n-1)}),$ we have
$$\begin{aligned}
\delta_{l(\t,n),1}+
\sum_{\substack{(\u,\v)\in\te{Std}^2(\P_{\ell,n})\\(\u,\v)\rhd (\s,\t)
\\\u_{\downarrow\leq (n-1)}=\s_{\downarrow\leq (n-1)}\\
\v_{\downarrow\leq (n-1)}=\t_{\downarrow\leq (n-1)}}}
\delta_{l(\u,n),1}c_{\u\v}^{\s\t}={\rm{c}}_{\t,n},\quad
\sum_{\substack{(\u,\v)\in\te{Std}^2(\P_{\ell,n})\\(\u,\v)\rhd (\s,\t)\\
\u_{\downarrow\leq (n-1)}=\fp\\
\v_{\downarrow\leq (n-1)}=\fq}} \delta_{l(\u,n),1}c_{\u\v}^{\s\t}={\rm{c}}_{\t,n} c_{\fp\fq}^{\s_{\downarrow\leq (n-1)}\t_{\downarrow\leq (n-1)}}.
\end{aligned} $$
  \item If $n(\s)\neq n(\t),$ then for each $(\fp,\fq)\in \Std^2(\P_{\ell,n-1})$
  with $(\fp,\fq)\rhd (\s_{\downarrow\leq (n-1)},\t_{\downarrow\leq (n-1)}),$ we have
$$\sum_{\substack{(\u,\v)\in\te{Std}^2(\P_{\ell,n})\\(\u,\v)\rhd (\s,\t)\\
\u_{\downarrow\leq (n-1)}=\fp\\
\v_{\downarrow\leq (n-1)}=\fq}} \delta_{l(\u,n),1}c_{\u\v}^{\s\t}=0. $$
\end{enumerate}
\end{cor}

\begin{proof}
This corollary follows from Propositions \ref{degvemst3}, \ref{degbasestranslationrela} and \ref{degvefst3},
and a similar argument as in the proof of Corollary \ref{maincor11}.
\begin{comment}
First, using Propositions \ref{degbasestranslationrela} and \ref{degvemst3}, we have
\begin{align}\label{degcompu12}
&\e_n(f_{\s\t})
=\e_n\Big(m_{\s\t}+\sum_{\substack{(\u,\v)\in\te{Std}^2(\P_{\ell,n})\\(\u,\v)\rhd (\s,\t)}}c_{\u\v}^{\s\t}m_{\u\v}\Big)\\
\notag =\ &\delta_{n(\s),n(\t)}\delta_{l(\t,n),1}m_{\s_{\downarrow\leq (n-1)}\t_{\downarrow\leq (n-1)}}
 +\sum_{\substack{(\u,\v)\in\te{Std}^2(\P_{\ell,n})\\(\u,\v)\rhd (\s,\t)}} \delta_{n(\u),n(\v)}\delta_{l(\u,n),1}c_{\u\v}^{\s\t}m_{\u_{\downarrow\leq (n-1)}\v_{\downarrow\leq (n-1)}}.
\end{align}
On the other hand, by Propositions \ref{degvefst3} and \ref{degbasestranslationrela}, we have
\begin{align}\label{degcompu13}
&\e_n(f_{\s\t})=\delta_{n(\s),n(\t)}{\rm{c}}_{\t,n}f_{\fs_{\downarrow\leq(n-1)}\ft_{\downarrow\leq(n-1)}}\\
\notag =\ &\delta_{n(\s),n(\t)} {\rm{c}}_{\t,n}\Big(m_{\s_{\downarrow\leq (n-1)}\t_{\downarrow\leq (n-1)}}
+\sum_{\substack{(\fp,\fq)\in\te{Std}^2(\P_{\ell,n-1})\\(\fp,\fq)\rhd (\s_{\downarrow\leq (n-1)},\t_{\downarrow\leq (n-1)})}}
c_{\fp\fq}^{\s_{\downarrow\leq (n-1)}\t_{\downarrow\leq (n-1)}}m_{\fp\fq}\Big).
\end{align}
Then the corollary follows from (\ref{degcompu12}) and (\ref{degcompu13}).
\end{comment}
\end{proof}

\begin{cor}
Let $\blam\in\P_{\ell,n}$ and $\s,\t\in\te{Std}(\blam).$
\begin{enumerate}
  \item If $n(\s)= n(\t),$ then
for each $(\fp,\fq)\in \Std^2(\P_{\ell,n-1})$ with $(\fp,\fq)\rhd (\s_{\downarrow\leq (n-1)},\t_{\downarrow\leq (n-1)}),$ we have
$$\begin{aligned}
&{\rm{c}}_{\t,n}+\sum_{\substack{(\u,\v)\in\te{Std}^2(\P_{\ell,n})\\(\u,\v)\rhd (\s,\t)
\\\u_{\downarrow\leq (n-1)}=\s_{\downarrow\leq (n-1)}\\
\v_{\downarrow\leq (n-1)}=\t_{\downarrow\leq (n-1)}}}
d_{\u\v}^{\s\t}{\rm{c}}_{\v,n}
=\delta_{l(\t,n),1},\quad
\sum_{\substack{(\u,\v)\in\te{Std}^2(\P_{\ell,n})\\(\u,\v)\rhd (\s,\t)\\
\u_{\downarrow\leq (n-1)}=\fp\\
\v_{\downarrow\leq (n-1)}=\fq}} d_{\u\v}^{\s\t} {\rm{c}}_{\v,n}=
\delta_{l(\t,n),1} d_{\fp\fq}^{\s_{\downarrow\leq (n-1)}\t_{\downarrow\leq (n-1)}}.
\end{aligned} $$
  \item If $n(\s)\neq n(\t),$ then for each $(\fp,\fq)\in \Std^2(\P_{\ell,n-1})$
  with $(\fp,\fq)\rhd (\s_{\downarrow\leq (n-1)},\t_{\downarrow\leq (n-1)}),$ we have
$$\sum_{\substack{(\u,\v)\in\te{Std}^2(\P_{\ell,n})\\(\u,\v)\rhd (\s,\t)\\
\u_{\downarrow\leq (n-1)}=\fp\\
\v_{\downarrow\leq (n-1)}=\fq}} d_{\u\v}^{\s\t} {\rm{c}}_{\v,n}=0. $$
\end{enumerate}
\end{cor}

\begin{proof}
This corollary follows from Propositions \ref{degvemst3}, \ref{degbasestranslationrela} and \ref{degvefst3},
and a similar argument as in the proof of Corollary \ref{maincor12}.
\begin{comment}
First, using Propositions \ref{degbasestranslationrela} and \ref{degvefst3}, we have
\begin{align}\label{degcompu14}
&\e_n(m_{\s\t})
=\e_n\Big(f_{\s\t}+\sum_{\substack{(\u,\v)\in\te{Std}^2(\P_{\ell,n})\\(\u,\v)\rhd (\s,\t)}}d_{\u\v}^{\s\t}f_{\u\v}\Big)\\
\notag =\ &\delta_{n(\s),n(\t)}{\rm{c}}_{\t,n}f_{\s_{\downarrow\leq (n-1)}\t_{\downarrow\leq (n-1)}}
+\sum_{\substack{(\u,\v)\in\te{Std}^2(\P_{\ell,n})\\(\u,\v)\rhd (\s,\t)}}
\delta_{n(\u),n(\v)}d_{\u\v}^{\s\t}{\rm{c}}_{\v,n}f_{\u_{\downarrow\leq (n-1)}\v_{\downarrow\leq (n-1)}}.
\end{align}
On the other hand, by Propositions \ref{degvemst3} and \ref{degbasestranslationrela}, we have
\begin{align}\label{degcompu15}
 &\e_n(m_{\s\t})=  \delta_{l(\t,n),1}\delta_{n(\s),n(\t)}m_{\s_{\downarrow\leq (n-1)}\t_{\downarrow\leq (n-1)}}\\
\notag =\ &\delta_{n(\s),n(\t)}\delta_{l(\t,n),1}
 \Big(f_{\s_{\downarrow\leq (n-1)}\t_{\downarrow\leq (n-1)}}
+\sum_{\substack{(\fp,\fq)\in\te{Std}^2(\P_{\ell,n-1})\\(\fp,\fq)\rhd (\s_{\downarrow\leq (n-1)},\t_{\downarrow\leq (n-1)})}}
d_{\fp\fq}^{\s_{\downarrow\leq (n-1)}\t_{\downarrow\leq (n-1)}}f_{\fp\fq}\Big).
\end{align}
Then the corollary follows from (\ref{degcompu14}) and (\ref{degcompu15}).
\end{comment}
\end{proof}

\begin{cor}
Let $\blam\in\P_{\ell,n}$ and $\s,\t\in\te{Std}(\blam).$ Then we have
\begin{align*}
\delta_{\s\t}\Biggr(\prod_{i=1}^{n}\delta_{l(\t,i),1}\Biggr)
+\sum_{\substack{\u\in\te{Std}(\P_{\ell,n})\\(\u,\u)\rhd (\s,\t)\\ l(\u,j)=1,\, \forall\,1\leq j\leq n}}
c_{\u\u}^{\s\t}=\delta_{\s\t} \mathrm{c}_\t,\quad
\delta_{\s\t}\mathrm{c}_\t+\sum_{\substack{\u\in\te{Std}(\P_{\ell,n})\\(\u,\u)\rhd (\s,\t)}}d_{\u\u}^{\s\t}\mathrm{c}_\u=\delta_{\s\t}\Biggr(\prod_{i=1}^{n}\delta_{l(\t,i),1}\Biggr).
\end{align*}
\end{cor}

\begin{proof}
This follows from Propositions \ref{degbasestranslationrela} and \ref{degvefst3} and Theorem \ref{mainthm2},
and a similar argument as in the proof of Corollary \ref{maincor13}.
\end{proof}

\end{document}